\documentclass[final]{siamltex}

\usepackage{amssymb,amsmath, amsfonts, amsbsy, amsbsy, latexsym, color, epsfig, graphicx, url, algorithm, algorithmic}
\usepackage{booktabs}
\usepackage{epstopdf}

\newcommand{\bitem}{\begin{itemize}}
\newcommand{\eitem}{\end{itemize}}
\newcommand{\beq}{\begin{equation}}
\newcommand{\eeq}{\end{equation}}

\newcommand{\cC}{{\cal C}}

\newcommand{\cS}{{\cal S}}
\newcommand{\cR}{{\cal R}}
\newcommand{\cH}{{\cal H}}
\newcommand{\cL}{{\cal L}}
\newcommand{\cD}{{\cal D}}

\newcommand{\bZ}{{\mathbb Z}}

\newcommand{\bR}{{\mathbb R}}
\newcommand{\bC}{{\mathbb C}}
\newcommand{\bN}{{\mathbb N}}

\newcommand{\ox}{\omega_x}
\newcommand{\oy}{\omega_y}

\newcommand{\CC}{\mathbb{C}}

\newcommand{\Sp}{\mbox{supp }}
\newcommand{\Id}{{\rm{Id }}}

\newcommand{\ip}[2]{\left\langle#1,#2\right\rangle}

\title{ShearLab: A Rational Design of a Digital Parabolic Scaling Algorithm}

\author{Gitta Kutyniok\thanks{Institute of Mathematics, University of Osnabr\"uck, 49069 Osnabr\"uck, Germany
({\tt kutyniok@uni-osnabrueck.de}).}
\and
Morteza Shahram\thanks{Department of Statistics, Stanford University, Stanford, CA 94305, USA
({\tt mshahram@stanford.edu}).}
\and
Xiaosheng Zhuang\thanks{Institute of Mathematics, University of Osnabr\"uck, 49069 Osnabr\"uck, Germany
({\tt xzhuang@uni-osnabrueck.de}).
\newline
\hspace*{0.4cm}
G.K. and M.S. would like to thank David Donoho for many inspiring discussions about
this work. They are also grateful to the Isaac Newton Institute of Mathematical Sciences in
Cambridge, UK for an inspiring research environment which led to the completion of a significant
part of this work. G.K. also thanks the Statistics Department at Stanford and the Mathematics
Department at Yale for hospitality and support during her visits. This work was partially supported
by Deutsche Forschungsgemeinschaft (DFG) Heisenberg fellowship KU 1446/8, DFG Grant SPP-1324
KU 1446/13, and DFG Grant KU 1446/14.}
}

\begin{document}

\maketitle

\begin{abstract}
Multivariate problems are typically governed by anisotropic
features such as edges in images. A common bracket of most of the various directional representation systems
which have been proposed to deliver sparse approximations of such features is the utilization of parabolic
 scaling. One prominent example is the shearlet system.
Our objective in this paper is three-fold: We firstly develop a digital shearlet theory which is rationally designed
in the sense that it is the digitization of the existing shearlet theory for continuous data. This implicates that
shearlet theory provides a unified treatment of both the continuum and digital realm. Secondly, we analyze
the utilization of pseudo-polar grids and the pseudo-polar Fourier transform for
digital implementations of parabolic scaling algorithms. We derive an isometric pseudo-polar Fourier
transform by careful weighting of the pseudo-polar grid, allowing exploitation of its adjoint for the inverse
transform. This leads to a digital implementation of the shearlet transform; an accompanying Matlab toolbox
called \url{ShearLab} is provided.
And, thirdly, we introduce various quantitative measures for digital parabolic scaling algorithms in general, allowing one
to tune parameters and objectively improve the implementation as well as compare different directional transform
implementations. The usefulness of such measures is exemplarily demonstrated for the digital shearlet transform.
\end{abstract}



\pagestyle{myheadings}
\thispagestyle{plain}
\markboth{G. KUTYNIOK, M. SHAHRAM, AND X. ZHUANG}{SHEARLAB: A RATIONAL DESIGN OF A PARABOLIC SCALING ALGORITHM}

\section{Introduction}
In recent years, applied harmonic analysts have introduced several approaches for directional representations of image data,
each one with the intent of efficiently representing highly anisotropic image features. Examples include curvelets
\cite{CD05a,CD05b,CDDY06}, contourlets \cite{DV05}, and shearlets \cite{GKL06,KLL10}. These proposals are
inspired by elegant results in theoretical harmonic analysis, which study functions defined on the continuum plane (i.e.,
not digital images) and address problems of efficiently representing certain types of functions and operators. One set
of inspiring results concerns the possibility of highly compressed representations of `cartoon' images, i.e., functions
which are piecewise smooth with singularities along smooth curves. Another set of results concerns the possibility of highly
compressed representations of wave propagation operators. In `continuum theory', anisotropic directional transforms can
significantly outperform wavelets in important ways.

Accordingly, one hopes that a digital implementation of such ideas would also deliver performance benefits over wavelet
algorithms in real-world settings. Anticipated applications include \cite{HH08}, where missing sensors cause incomplete
measurements, and the problem of texture/geometry separation in image processing -- for example in astronomy when images
of galaxies require separated analyses of stars, filaments, and sheets \cite{SCD03,DK10}.

In many cases, however, there are no publicly available implementations of such ideas, or the available implementations are
only sketchily tested, or the available implementations are only vaguely related to the continuum transforms they are
reputed to represent.  Accordingly, we have not yet seen a serious exploration of the potential benefit of such transforms,
carefully comparing the expected benefits with those delivered by specific implementations.

In this paper we aim at providing both:
\bitem
\item[(1)] A rationally designed shearlet transform implementation.
\item[(2)] A comprehensive framework for quantifying performance of directional representations in general.
\eitem

For (1), we developed an implementation of the fast digital shearlet transform (FDST) based on a digital shearlet theory which is a very
natural digitization of the existing shearlet theory for continuous data. Other parabolic-scaling transforms, for example, curvelets
are inherently based on operations (rotation) which translate awkwardly into the digital realm. In
contrast, when we consider shearlets, rotations are replaced by shearing, which has a natural digital realization, thus
enabling a unified treatment for the continuum and digital realm similar to wavelets.

The framework in (2) has three benefits. First, it provides quantitative performance measures which can be used to tune the parameters
of our implementation, which is publicly available at \url{www.ShearLab.org}. This allows us to specify `recommended choices'
for the parameters of our implementation. Second, the
same `measure and tune' approach may be useful to other implementers of directional transforms. Third, we show a way to improve
the level of intellectual seriousness in applied mathematics which pretends to work in image processing. We believe that widespread
adoption of this measure and tune framework can be very valuable, since many supposedly scientific presentations are now little
more than vague, numbing  `advertising' or `marketing' pitches.  They could instead offer quantitative comparisons between
algorithms, and thereby be far more informative. In fact the combination of quantitative evaluation with reproducible research
\cite{DMSSU08} would be particularly effective at producing both intellectual seriousness and rapid progress.

\subsection{Desiderata}
\label{subsec:desiderata}

We start by proposing the following desiderata for the fast digital shearlet transform FDST and its implementation:

\renewcommand{\labelenumi}{{\rm [D\arabic{enumi}]}}

\begin{enumerate}
\item {\em Algebraic Exactness.} The transform should be based on a shearlet theory
for digital data on a pseudo-polar grid, than merely being `somewhat
close' to the shearlet theory for continuous data.
\item {\em Isometry of Pseudo-Polar Fourier Transform.} We intro\-duce over\-samp\-ling and
weights to ob\-tain an isometric pseudo-polar Fourier transform, which allows us to use
the adjoint as inverse transform.
\item {\em Tight Frame Property.} The shearlet coefficients computed by the transform
should be based on a tight frame decomposition, which ensures an
isometric relation between the input image and the sequence of
coefficients as well as allows us to use the adjoint as inverse
transform. This property follows by combining [D1] and [D2], and
allows the comparison with other transforms in contrast to those previous two tests.
\item {\em Time-Frequency-Localization.} The spatial portrait of the analyzing elements
should `look like' shearlets in the sense that they are sufficiently
smooth as well as time-localized. Localization and smoothness in frequency domain is
ensured by definition.
\item {\em True Shear Invariance.} Since the orientation-related operator of shearlets
is in fact the shear operator, we expect to see a shearing of the input image mirrored in
a simple shift of the transform coefficients.
\item {\em Speed.} The transform should admit an al\-go\-ri\-thm of or\-der $O(N^2 \log N)$
flops, where $N^2$ is the number of digital points of the input image.
\item {\em Geometric Exactness.} The transform should preserve geometric properties
parallel to those of the continuum theory, for example, edges should be mapped to
edges in shearlet domain.
\item {\em Robustness.} The transform should be resilient against impacts such as
(hard) thresholded and quantized coefficients.
\end{enumerate}


\subsection{Definition of the Shearlet Transform}
\label{subsec:defST}

The main idea for the construction of the shearlet transform with discrete parameters for functions in $L^2(\bR^2)$
is the choice of a two-parameter dilation group, where one parameter ensures the multiscale property, whereas the
second parameter provides a means to detect directions. The choice for a direction sensitive parameter is particularly
important, since the most canonical choice, the rotation, would prohibit a unified treatment of the continuum and
digital realm due to the fact that the integer grid is not invariant under rotation. Shearlets parameterize directions by slope
rather than angles. And the shear matrix does preserve the structure of the integer grid, which is key to enabling
an exact digitization of the continuum domain shearlets.

For each $a>0$ and $s \in \mathbb{R}$, let $A_a$ denote the {\em parabolic scaling matrix} and $S_s$ denote the
{\em shear matrix} of the form
\[
A_a =
\begin{pmatrix}
  a & 0\\ 0 & \sqrt{a}
\end{pmatrix}
\qquad\mbox{and}\qquad
S_s = \begin{pmatrix}
  1 & s\\ 0 & 1
\end{pmatrix},
\]
respectively. To provide an equal treatment of the $x$- and
$y$-axis, the frequency plane is split into the four cones
$\cC_{11}$ -- $\cC_{22}$ (see  Figure~\ref{fig:ShearletsCone}),
defined by
\[
\cC_\iota = \left\{ \begin{array}{rcl}
\{(\xi_1,\xi_2) \in \bR^2 : \xi_1 \ge 1,\;\;\; |\xi_1/\xi_2| \ge 1\} & : & \iota = 21,\\
\{(\xi_1,\xi_2) \in \bR^2 : \xi_2 \ge 1,\;\;\; |\xi_1/\xi_2| \le 1\} & : & \iota = 11,\\
\{(\xi_1,\xi_2) \in \bR^2 : \xi_1 \le -1,\, |\xi_1/\xi_2| \ge 1\} & : & \iota = 22,\\
\{(\xi_1,\xi_2) \in \bR^2 : \xi_2 \le -1,\, |\xi_1/\xi_2| \le 1\} & : & \iota = 12.
\end{array}
\right.
\]
Let now $\psi_1 \in L^2(\bR)$ be a wavelet with $\hat{\psi}_1 \in C^\infty(\mathbb{R})$
and supp $\hat{\psi}_1 \subseteq [-4,-\frac14] \cup [\frac14,4]$, and let
$\psi_2 \in L^2(\mathbb{R})$ be a `bump' function satisfying $\hat{\psi}_2 \in C^\infty(\mathbb{R})$ and
supp $\hat{\psi}_2 \subseteq [-1,1]$. We define $\psi \in L^2(\mathbb{R}^2)$ by
\beq \label{eq:psidef}
\hat{\psi}(\xi) = \hat{\psi}(\xi_1,\xi_2) = \hat{\psi}_1(\xi_1) \, \hat{\psi}_2(\tfrac{\xi_2}{\xi_1}).
\eeq
For cone $\cC_{21}$, at scale $j\in\bN_0:=\bN\cup\{0\}$, orientation $s = -2^{j}, \dots,$ $2^{j}$,
and spatial position $m \in \bZ^2$,  the associated {\em shearlets} are then defined by
their Fourier transforms
\begin{eqnarray} \nonumber
\hat{\sigma}_{\eta}(\xi)
& = & 2^{-j\frac{3}{2}} \hat{\psi}(S_s^T A_{4^{-j}} \xi) \chi_{\cC_{21}}(\xi) e^{2\pi i\ip{A_{4^{-j}} S_s m}{\xi}}\\ \label{eq:shearlets}
& = & 2^{-j\frac{3}{2}} \hat{\psi}_1({4^{-j}}{\xi_1}) \hat{\psi}_2(s + 2^{j}\tfrac{\xi_2}{\xi_1}) \chi_{\cC_{21}}(\xi) e^{2\pi i \ip{A_{4^{-j}} S_s m}{\xi}},
\end{eqnarray}
where $\eta = (j,s,m,\iota)$ index scale, orientation,  position, and cone.
The shearlets for $\cC_{11}$, $\cC_{21}$, and $\cC_{22}$ are defined
likewise by symmetry, as illustrated in Figure
\ref{fig:ShearletsCone}, and we denote the resulting {\em discrete
shearlet system} by \beq \label{eq:shearletsystem} \{\sigma_{\eta} :
\eta \in \bN_0 \times \{-2^{j}, \dots, 2^{j}\} \times \bZ^2 \times
\{11, 12, 21, 22\} \}. \eeq The definition shows that shearlets live
on anisotropic regions of width $2^{-2j}$ and length $2^{-j}$ at
various orientations.

It should be mentioned that discrete shearlets -- `discrete' referring to the set of parameters
and {\em not} to the domain -- can also be defined with respect to the dilation matrix $A_{2^{-j}}$. However, in this case the odd scales
have to be handled particularly carefully.  The attentive reader will have also observed that recently
introduced compactly supported shearlets \cite{KKL10,KLL10} do not require projecting the shearlets to the
respective cones; however, despite other advantageous properties, they do not form a tight frame for $L^2(\bR^2)$.
Finally, the generating window allows in fact more freedom than \eqref{eq:psidef}, but in this paper we
restrict ourselves to this (customary) choice.
\begin{figure}[ht]
\begin{center}
\scalebox{1.0}{\includegraphics[height=1.3in]{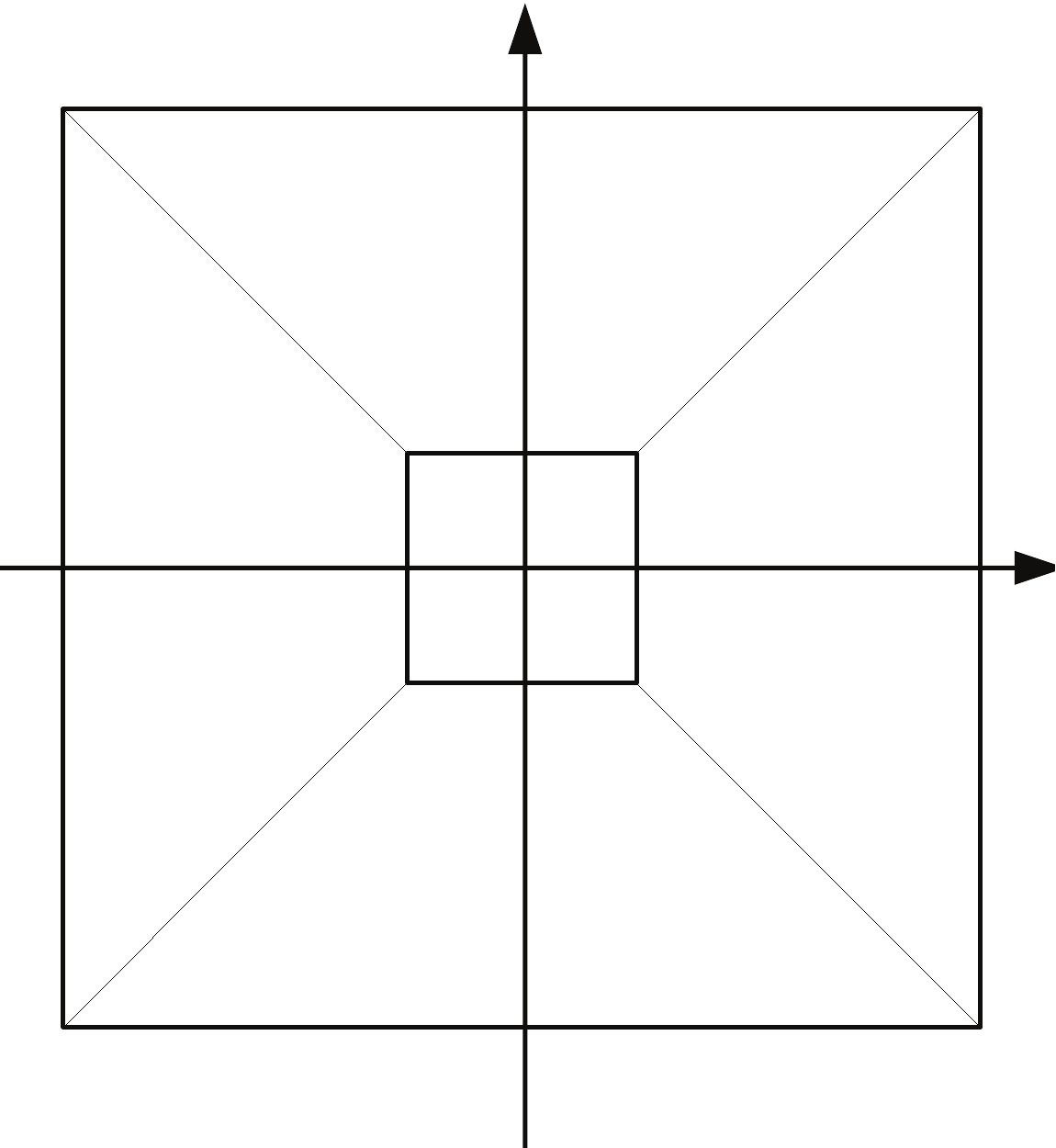}}
\put(-22,54){$\mathcal{C}_{21}$}
\put(-40,72){$\mathcal{C}_{11}$}
\put(-76,37){$\mathcal{C}_{22}$}
\put(-59,20){$\mathcal{C}_{12}$}
\hspace*{1.5cm}
\scalebox{1.0}{\includegraphics[height=1.35in]{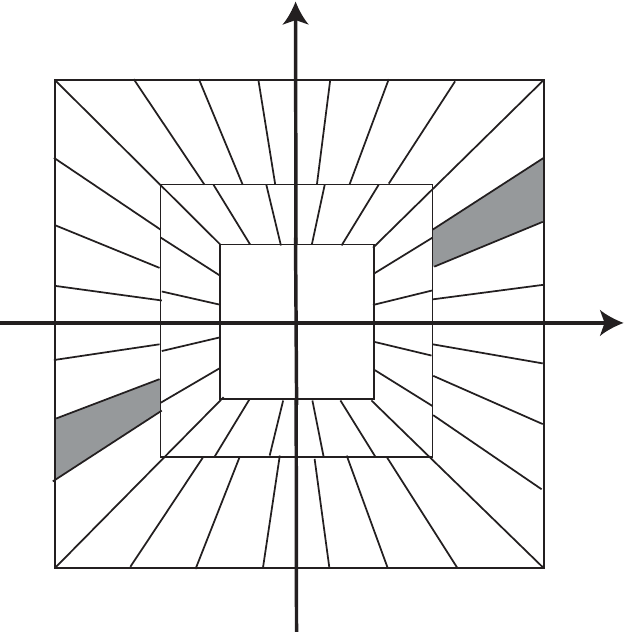}}
\end{center}
\caption{The cones $\cC_{11}$ -- $\cC_{22}$ and the tiling of the frequency domain induced by shearlets.}
\label{fig:ShearletsCone}
\end{figure}

Setting $\cC^\vee = \bigcup_{\iota = 11}^{22} \cC_\iota$,
we have the following theorem from \cite[Thm. 3]{GKL06} concerning the frame properties of the
discrete shearlet system. For the definition of a tight (sometimes called Parseval) frame, we refer to \cite{Chr03}.
\begin{theorem}[\cite{GKL06}]
The system \eqref{eq:shearletsystem} is a tight frame for $\{f \in L^2(\bR^2) : \Sp \hat{f} \subseteq \cC^\vee\}$.
\end{theorem}
We remark that the low frequency part can be appropriately filled in to obtain a tight frame for $L^2(\bR^2)$.

The transform associated with this system is the {\em discrete shearlet transform}, which for a given function
$f \in L^2(\bR^2)$ is defined to be the map
\[
\bN_0 \times \{-2^{j}, \dots, 2^{j}\} \times \bZ^2 \times \{11, 12, 21, 22\} \ni \eta \mapsto
(\ip{f}{\sigma_\eta}) \in \bC.
\]
It is this transform, which we aim to exactly digitize.


\subsection{Ingredients of the Fast Digital Shearlet Transform (FDST)}
\label{subsec:DST}

The shearlet transform for continuum domain data (see Figure
\ref{fig:ShearletsCone}) implicitly induces a trapezoidal tiling of
frequency space which is evidently not cartesian. By introducing a
special set of coordinates on the continuum 2D frequency space, the
discrete shearlet transform can be represented as a cascade of five
operations: \bitem
\item Classical Fourier transformation.
\item Change of variables to pseudo-polar coordinates.
\item Weighting by a radial `density com\-pen\-sa\-tion' factor.
\item Decomposition into rectangular tiles.
\item Inverse Fourier transform of each tiles.
\eitem

Surprisingly, this process admits a natural translation into the
digital domain. The key observation is that the pseudo-polar
coordinates are naturally compatible with digital image processing
(compare Figure \ref{fig:PPgrid}) and perfectly suited for a
digitization of the discrete shearlet transform as a comparison with
the frequency tiling generated by continuum domain shearlets in
Figure \ref{fig:ShearletsCone} already visually evidences.
Fortunately, in \cite{ACDIS08} a fast pseudo-polar Fourier transform
(PPFT) is already developed. This transform evaluates the Fourier
transform of an image of size $N$, say, on a pseudo-polar grid
$\Omega$ of the form
$\Omega = \Omega^1 \cup \Omega^2$,
where
\begin{eqnarray*}
\Omega^1 & = &  \{(- k\cdot\tfrac{2\ell}{N},k) : -\tfrac{N}{2}\le \ell \le \tfrac{N}{2}, \, -N \le k \le N\},\\
\Omega^2 & = &  \{(k,-k\cdot\tfrac{2\ell}{N}) : -\tfrac{N}{2} \le \ell \le \tfrac{N}{2}, \, -N \le k \le N\}.
\end{eqnarray*}
Figure \ref{fig:PPgrid} shows an illustration of the case $N=4$.
\begin{figure}[ht]
\begin{center}
\includegraphics[height=1.0in]{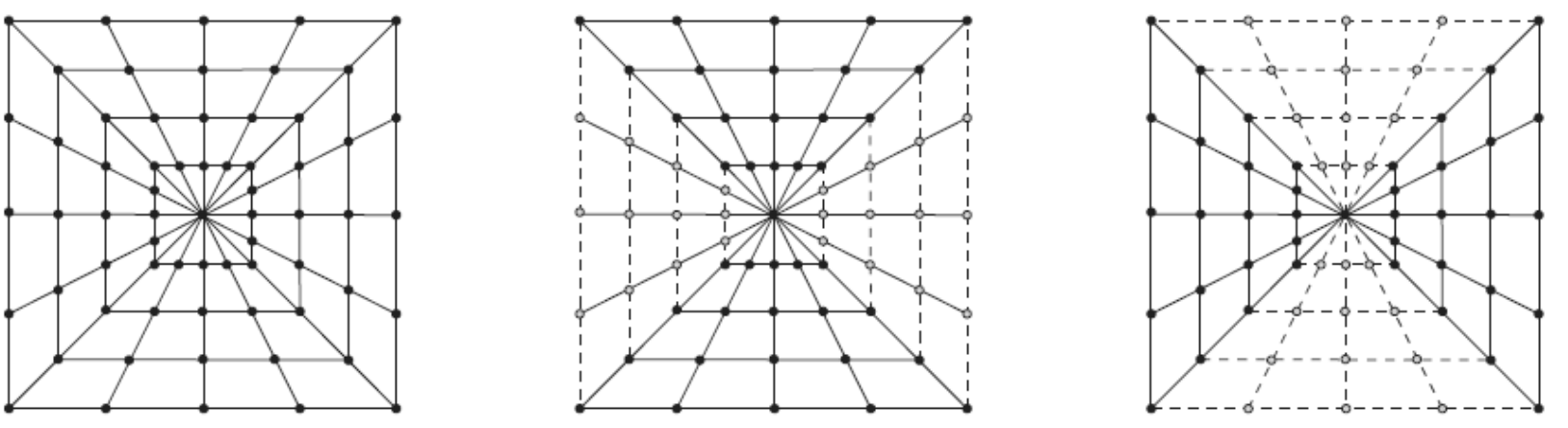}
\end{center}
\caption{The pseudo-polar grid $\Omega=\Omega^1\cup\Omega^2$ for $N=4$.}
\label{fig:PPgrid}
\end{figure}
For an $N\times N$ image $I:=\{I(u,v) : -N/2\le u,v \le N/2-1\}$, the pseudo-polar Fourier
transform $\hat{I}$ of $I$ evaluated on the pseudo-polar grid $\Omega$ is then defined to be
\[
\hat{I}(\ox,\oy) = \sum_{u,v=-N/2}^{N/2-1}I(u,v)e^{-\frac{2\pi i}{m_0}(u\ox+v\oy)}, \quad (\ox,\oy)\in\Omega,
\]
where $m_0 \ge N$ is an integer which for the PPFT is chosen to be $m_0 = 2N+1$ for computational reasons.

The existence of PPFT suggests that we can easily and naturally get a faithful FDST using this algorithm.
However, besides the delicateness of digitizing the continuum domain shearlets so that they form a tight
frame on the pseudo-polar grid, also the use of the PPFT is not at all straightforward. The PPFT as presented
\cite{ACDIS08} is not an isometry. The main obstacle is the highly nonuniform arrangement of the points on
the pseudo-polar grid. This intuitively suggests to downweight points in regions of very high density
by using weights which correspond roughly to the density compensation  weights underlying the continuous
change of variables. In fact, we will show that isometry is possible with sufficient radial oversampling
of the pseudo-polar grid; however, the weights will not be derivable from simple density compensation
arguments.

Summarizing, the FDST of an $N \times N$ image cascades the following steps:
\bitem
\item[1)] Application of the PPFT with an oversampling factor $R$ in radial direction.
\item[2)] Weighting of the function values on the pseudo-polar grid by `density-com\-pen\-sa\-tion-style' weights.
\item[3)] Decomposing the pseudo-polar-indexed values by a scaled and sheared generating window into
rectangular subbands followed by application of the 2D iFFT to each array.
\eitem
This is an exact analogy of the discrete shearlet transform, in which the steps of Fourier transformation and pseudo-polar coordinate
change as well as the steps of decomposition into rectangular tiles and the inverse Fourier transform are collapsed into one step,
respectively. With a careful choice of the weights and the windows, this transform is an
isometry as we will show. Hence the inverse transform can be computed by merely taking the adjoint in each
step.


\subsection{Performance Measurement}

The above sketch does not uniquely specify an implementation; there is freedom in choice of weights and windows.
How can we decide if one choice is better than another one? It seems that currently researchers often use
overall system performance on isolated tasks, such as denoising and compression of specific standard images like
`Lena', `Barbara', etc. However, overall system performance for a system made up of a cascade of steps seems very
opaque and at the same time very particular. It seems far better from an intellectual viewpoint to carefully
decompose performance according to a more insightful array of tests, each one motivated by a particular
well-understood property we are trying to obtain.

We have developed  quantitative performance measures inspired by the desiderata we presented in Subsection~\ref{subsec:desiderata}.
Each performance measure produces a real value or a real-valued curve, thus providing a standardized framework for
evaluation and, especially, comparison.


\subsection{Relation with Previous Work}

Since the introduction of directional representation systems by many pioneer researchers (\cite{CD99,CD04,CD05a,CD05b,DV05,GKL06}),
various numerical implementations of their directional representation systems have been proposed. The closest ones are the curvelet,
contourlet, and previous shearlet algorithms, whose main features we now briefly survey.


{\em Curvelets} \cite{CDDY06}. The discrete curvelet transform is implemented in the software package {\em CurveLab}, which
comprises two different approaches. One is based on unequispaced FFTs, which are used to interpolate the function in the frequency
domain on different tiles with respect to different orientations of curvelets. The other is based on frequency wrapping, which wraps
each subband indexed by scale and angle into a fixed rectangle around the origin. Both approaches can be realized efficiently in
$O(N^2\log N)$ flops with $N$ being the image size. The disadvantage of this approach is the lack of an associated continuum domain
theory.

{\em Contourlets} \cite{DV05}. The implementation of contourlets is based on a directional filter bank, which produces a
directional frequency partitioning similar to the one generated by curvelets. The main advantage of this approach is that it allows
a tree-structured filter bank implementation, in which aliasing due to subsampling is allowed to exist. Consequently, one can
achieve great efficiency in terms of redundancy and good spatial localization. A drawback of this approach is that various
artifacts are introduced and that an associated continuum domain theory is missing.

{\em Shearlets} \cite{ELL08, Lim2010}.
In \cite{ELL08}, Easley et. al.
implemented the shearlet transform by applying the Laplacian pyramid scheme and directional filtering successionally.
One drawback is the deviation from the continuum domain theory. Another drawback is that the associated code was not made publicly
available. In contrast to this implementation which is based on bandlimited subband tiling --  similar to the implementation of
curvelets in \cite{CDDY06} -- in \cite{Lim2010}, Lim provided an implementation of the shearlet transform based on compactly
supported shearlet systems (see also \cite{KKL10}). These compactly supported shearlets are separable and provide excellent spatial
localization. The drawback is that they do not form a tight frame, hence, the synthesis process needs to be performed by iterative
methods. We further wish to mention two novel approaches \cite{KS09} and  \cite{HKZ10} for which however no implementation is yet
available nor was their focus on deriving an exact digitization of the continuum domain transform.


Summarizing, all the above implementations of directional representation systems have their own advantages and disadvantages, one
of the most common shortcomings being the lack of providing the unified treatment of the continuum and digital world.
Our effort will now be put to provide a natural digitization of the shearlet theory (bandlimited shearlets) fulfilling the unified
treatment requirement as well as a software package \url{ShearLab} quantifying performances of directional representation systems.


\subsection{Contribution of this Paper}

The contributions of this paper are two-fold. Firstly, we introduce a fast digital shearlet transform (FDST) which is {\em rationally designed}
based on a {\em natural} digitization of shearlet theory. Secondly, we provide a variety of {\em quantitative performance measures}
for directional representations, which allow tuning and comparison of implementations. Our digital shearlet implementation was tuned
utilizing this framework, so we can provide the user community with an optimized representation.

All presented algorithms and tests are provided at \url{www.ShearLab.org}
in the spirit of reproducible research \cite{DMSSU08}.


\subsection{Contents}

Section \ref{sec:STfinitedata} introduces the fast digital shearlet transform FDST and proves isometry. In Section \ref{sec:ISTfinitedata}, we
then discuss two variants of an inverse digital shearlet transform, namely, a direct and an iterative approach.
In Section \ref{sec:props}, we prove several mathematical properties of the FDST such as decay properties
of digital shearlet coefficients. The following section, Section \ref{sec:implementation}, is concerned with details of the associated
\url{ShearLab} implementation at \url{www.ShearLab.org}. The FDST is then analyzed in Section \ref{sec:numerics} according to the
quantitative measures introduced in Section \ref{sec:qualitymeasures}.


\section{FDST for Finite Data}
\label{sec:STfinitedata}

We start by discussing the three steps in the FDST as described in Subsection \ref{subsec:DST},  which we for the
convenience of the reader briefly repeat:
%
%
\bitem
\item[1)] Application of the PPFT with an oversampling factor $R$ in radial direction.
\item[2)] Weighting of the function values on the pseudo-polar grid by `density-com\-pen\-sa\-tion-style' weights.
\item[3)] Decomposing the pseudo-polar-indexed values by a scaled and sheared generating window into
rectangular subbands followed by application of the 2D iFFT to each array.
\eitem

We will also show that careful selection of the oversampling factor, of the weights, and of the windows yields an
isometric transform, which enables us to compute the inverse shearlet transform by its adjoint
(see Section \ref{sec:ISTfinitedata}).


\subsection{Weighted Pseudo-Polar Fourier Transform}
\label{subsec:wPPFT}

Given an $N \times N$ image $I$, it is well known that the Fourier transform $\hat{I}$ of $I$ evaluated on a rectangular
$N \times N$ grid is an isometry:
\beq \label{eq:Plancherel}
\sum_{u, v = -N/2}^{N/2-1}  |I(u,v)|^2 = \frac{1}{N^2} \sum_{\ox, \oy = -N/2}^{N/2-1} |\hat{I}(\ox,\oy)|^2.
\eeq
This is the Plancherel formula for a function defined on a finite group \cite{HR63}.

We now intend to obtain a similar formula for the Fourier transform of $I$ evaluated on the pseudo-polar grid. For this,
we first extend the definition of the pseudo-polar grid slightly by introducing an oversampling parameter $R > 0$ in
radial direction. This new grid, which we will denote in the sequel by $\Omega_R$, is defined by
\[
\Omega_R = \Omega_R^1 \cup \Omega_R^2,
\]
where
\begin{eqnarray} \label{eq:OmegaR1}
\Omega_R^1 & = &  \{(-\tfrac{2k}{R}\cdot\tfrac{2\ell}{N},\tfrac{2k}{R}) : -\tfrac{N}{2} \le \ell \le \tfrac{N}{2}, \, -\tfrac{RN}{2} \le k \le \tfrac{RN}{2}\},\\ \label{eq:OmegaR2}
\Omega_R^2 & = &  \{(\tfrac{2k}{R},-\tfrac{2k}{R}\cdot\tfrac{2\ell}{N}) : -\tfrac{N}{2} \le \ell \le \tfrac{N}{2}, \, -\tfrac{RN}{2} \le k \le \tfrac{RN}{2}\}.
\end{eqnarray}
This grid is illustrated in Figure \ref{fig:PPgridR}.
\begin{figure}[ht]
\begin{center}
\includegraphics[height=1.2in]{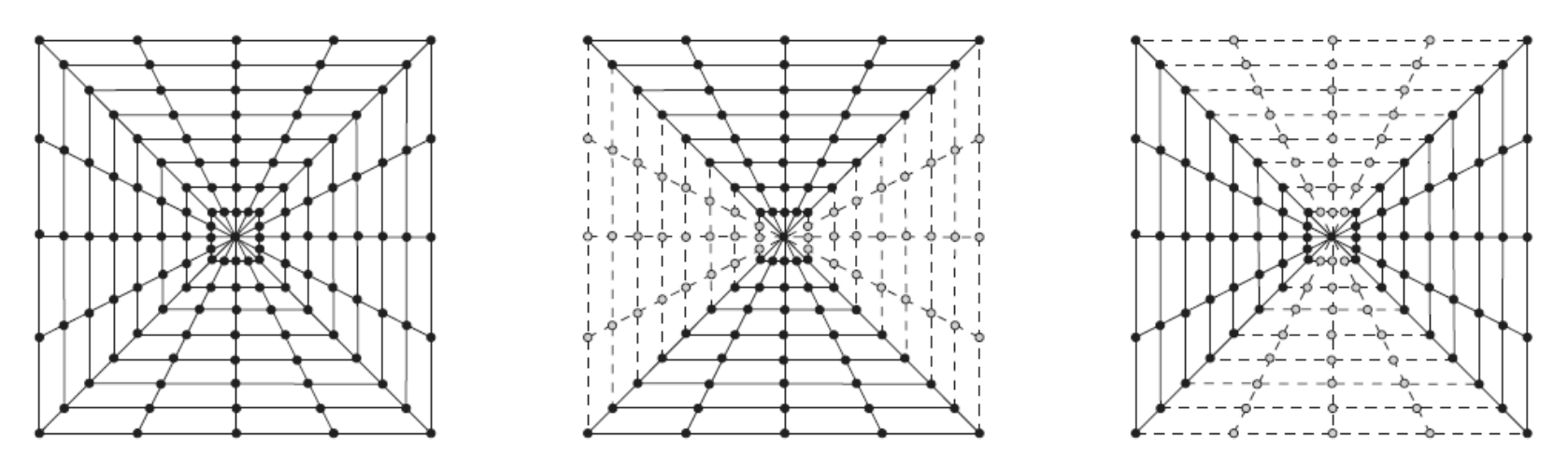}
\end{center}
\caption{The pseudo-polar grid $\Omega_R = \Omega_R^1 \cup \Omega_R^2$ for $N=4$ and $R=4$.}
\label{fig:PPgridR}
\end{figure}
Notice that the `original' pseudo-polar grid (see Figure~\ref{fig:PPgrid}) as introduced in Subsection \ref{subsec:DST}
is a special case of this definition when choosing $R=2$. Also observe that the center
\[
\cC = \{(0,0)\}
\]
appears $N+1$ times in $\Omega_R^1$ as well as $\Omega_R^2$, and the points on the seam lines
\begin{eqnarray*}
\cS_R^1 & = &  \{(-\tfrac{2k}{R},\tfrac{2k}{R}) : -\tfrac{RN}{2} \le k \le \tfrac{RN}{2}, \, k \neq 0\},\\
\cS_R^2 & = &  \{(\tfrac{2k}{R},-\tfrac{2k}{R}) : -\tfrac{RN}{2} \le k \le \tfrac{RN}{2}, \, k \neq 0\},
\end{eqnarray*}
appear in both $\Omega_R^1$ and $\Omega_R^2$. Later, we will also utilize a further partitioning of the sets $\Omega_R^1$
and $\Omega_R^2$ as
\[
\Omega_R^1 = \Omega_R^{11} \cup \cC \cup \Omega_R^{12}
\qquad \mbox{and} \qquad
\Omega_R^2 = \Omega_R^{21} \cup \cC \cup \Omega_R^{22},
\]
where
\begin{eqnarray*}
\Omega_R^{11} & = &  \{(-\tfrac{2k}{R}\cdot\tfrac{2\ell}{N},\tfrac{2k}{R}) : -\tfrac{N}{2} \le \ell \le \tfrac{N}{2}, \, 1 \le k \le \tfrac{RN}{2}\},\\
\Omega_R^{12} & = &  \{(-\tfrac{2k}{R}\cdot\tfrac{2\ell}{N},\tfrac{2k}{R}) : -\tfrac{N}{2} \le \ell \le \tfrac{N}{2}, \, -\tfrac{RN}{2} \le k \le -1\},\\
\Omega_R^{21} & = &  \{(\tfrac{2k}{R},-\tfrac{2k}{R}\cdot\tfrac{2\ell}{N}) : -\tfrac{N}{2} \le \ell \le \tfrac{N}{2}, \, 1 \le k \le \tfrac{RN}{2}\}\\
\Omega_R^{22} & = &  \{(\tfrac{2k}{R},-\tfrac{2k}{R}\cdot\tfrac{2\ell}{N}) : -\tfrac{N}{2} \le \ell \le \tfrac{N}{2}, \, -\tfrac{RN}{2} \le k \le -1\}.
\end{eqnarray*}

Now our goal is to choose weights $w : \Omega_R \to \bR^+$ so that, for any $N \times N$ image $I$,
\beq \label{eq:ppPlancherel}
\sum_{u, v = -N/2}^{N/2-1} |I(u,v)|^2
= \sum_{(\ox, \oy) \in \Omega_R} w(\ox,\oy) \cdot |\hat{I}(\ox,\oy)|^2,
\eeq
where here we modify the definition of the Fourier transform according to \cite{ACDIS08} and define it by
\begin{equation}\label{def:ppft}
\hat{I}(\ox,\oy) = \sum_{u, v = -N/2}^{N/2-1}  I(u,v) e^{-\frac{2\pi i}{m_0}(u \ox + v \oy)},\quad (\ox,\oy)\in\Omega_R,
\end{equation}
where $m_0 \ge N$. Also notice that the factor $1/N^2$ appearing in \eqref{eq:Plancherel} will now be hidden in the weights $w(\ox,\oy)$.

We start by computing the right hand side of \eqref{eq:ppPlancherel}:
\begin{eqnarray*}
\lefteqn{\sum_{(\ox, \oy) \in \Omega_R} w(\ox,\oy) \cdot |\hat{I}(\ox,\oy)|^2}\\
& = & \sum_{(\ox, \oy) \in \Omega_R} w(\ox,\oy)
\cdot \left|\sum_{u, v = -N/2}^{N/2-1}  I(u,v) e^{-\frac{2\pi i}{m_0}(u \ox + v \oy)}\right|^2\\
& = & \sum_{(\ox, \oy) \in \Omega_R} w(\ox,\oy)
\cdot \left[\sum_{u, v = -N/2}^{N/2-1} \sum_{u', v' = -N/2}^{N/2-1} \hspace*{-0.18cm} I(u,v) \overline{I(u',v')}
e^{-\frac{2\pi i}{m_0}((u-u') \ox + (v-v') \oy)}\right]\\
& = & \sum_{(\ox, \oy) \in \Omega_R} w(\ox,\oy) \cdot  \sum_{u, v = -N/2}^{N/2-1} |I(u,v)|^2\\
& & + \sum_{\stackrel{u, v, u', v' = -N/2}{(u,v) \neq (u',v')}}^{N/2-1} I(u,v) \overline{I(u',v')} \cdot
\left[ \sum_{(\ox, \oy) \in \Omega_R} w(\ox,\oy) \cdot e^{-\frac{2\pi i}{m_0}((u-u') \ox + (v-v') \oy)}
\right].
\end{eqnarray*}
Choosing $I=c_{u_1,v_1}\delta{(u-u_1,v-v_1)}+c_{u_2,v_2}\delta{(u-u_2,v-v_2)}$ for all $-N/2\le u_1,$ $v_1,$ $u_2,$ $v_2\le N/2-1$
and for all $c_{u_1,v_1},c_{u_2,v_2}\in\CC$, we can conclude that \eqref{eq:ppPlancherel} holds if and only if
\begin{equation}\label{eq:isometry0}
 \sum_{(\ox, \oy) \in \Omega_R} w(\ox,\oy) \cdot e^{-\frac{2\pi i}{m_0}(u\ox + v\oy)} =
  \delta(u,v),\quad -N+1 \le u, v \le N-1.
\end{equation}
This is equivalent to the two conditions
\begin{eqnarray}
\label{eq:isometry01}
\sum_{(\ox, \oy) \in \Omega_R} w(\ox,\oy) \cdot \cos(\tfrac{2\pi}{m_0} (u \ox + v \oy))
&=& \delta(u,v), 
\\ \label{eq:isometry02}
\sum_{(\ox, \oy) \in \Omega_R} w(\ox,\oy) \cdot \sin(\tfrac{2\pi}{m_0} (u \ox + v \oy))
&=&0,
\end{eqnarray}
for all $-N+1 \le u, v \le N-1$. In view of the symmetry of the pseudo-polar grid, it is natural to impose the
following symmetry conditions on the weights:
\renewcommand{\labelenumi}{{\rm [S\arabic{enumi}]}}
\begin{enumerate}
\item $w(\ox, \oy) = w(\oy, \ox)$,\;\;\, $(\ox, \oy) \in \Omega_R$,
\item $w(\ox, \oy) = w(-\oy, \ox)$, $(\ox, \oy) \in \Omega_R$,
\item $w(\ox, \oy) = w(-\ox, \oy)$, $(\ox, \oy) \in \Omega_R$,
\item $w(\ox, \oy) = w(\ox, -\oy)$, $(\ox, \oy) \in \Omega_R$.
\end{enumerate}
In this case, \eqref{eq:isometry02} automatically holds. By the sum formula for trigonometric functions,  \eqref{eq:isometry01} is
then equivalent to
\[
\sum_{(\ox, \oy) \in \Omega_R} w(\ox,\oy) \cdot
[\cos(\tfrac{2\pi}{m_0} u \ox)\cos(\tfrac{2\pi}{m_0} v \oy)-\sin(\tfrac{2\pi}{m_0} u \ox)\sin(\tfrac{2\pi}{m_0} v \oy)] = \delta(u,v)
\]
 for all $-N+1 \le u, v \le N-1$.
Again, by the symmetry of the weights, this is equivalent to
\begin{equation}\label{eq:isometry1}
\sum_{(\ox, \oy) \in \Omega_R} w(\ox,\oy) \cdot
[\cos(\tfrac{2\pi}{m_0} u \ox)\cos(\tfrac{2\pi}{m_0} v \oy)] = \delta(u,v)
\end{equation}
for all $-N+1\le u,v\le N-1$. This is a linear system of equations with $RN^2/4+RN/2+1$ unknows and $(2N-1)^2$ equations, wherefore, in general, we need the
oversampling factor $R$ to be at least $16$ to enforce solvability.

For symmetry reasons, we can now restrict our attention to one quarter of a cone, say $\Omega_R^{21}$. Using \eqref{eq:OmegaR1},
\eqref{eq:OmegaR2}, and \eqref{eq:isometry1}, we then obtain the following equivalent condition to \eqref{eq:isometry0}:
\begin{eqnarray} \nonumber
\delta(u,v) & \hspace*{-0.05cm} = \hspace*{-0.05cm} & w(0,0) + 4 \cdot \hspace*{-0.1cm} \sum_{\ell = 0, N/2} \sum_{k = 1}^{RN/2}
w(\tfrac{2k}{R}, -\tfrac{2k}{R}\cdot\tfrac{2\ell}{N}) \cdot \cos(2\pi  u\cdot \tfrac{2k}{m_0R})\cdot \cos(2\pi v \cdot \tfrac{2k}{m_0R}\cdot\tfrac{2\ell}{N})\\  \label{eq:weightcondition}
& & + 8 \cdot \sum_{\ell = 1}^{N/2-1} \sum_{k = 1}^{RN/2}
w(\tfrac{2k}{R}, -\tfrac{2k}{R}\cdot\tfrac{2\ell}{N})  \cdot \cos(2\pi  u\cdot \tfrac{2k}{m_0R})\cdot \cos(2\pi v \cdot \tfrac{2k}{m_0R}\cdot\tfrac{2\ell}{N})
\end{eqnarray}
for all $-N+1\le u, v \le N-1$. Concluding, we have the following result.

\begin{theorem} \label{theo:weightedisometry}
Let $N$ be even, let $\Omega_R = \Omega_R^1 \cup \Omega_R^2$ be the pseudo-polar grid defined in \eqref{eq:OmegaR1} and
\eqref{eq:OmegaR2}, and let $w : \Omega_R \to \bR^+$ be a weight function satisfying the symmetry conditions {\rm [S1]} --
{\rm [S4]}. Then
\[
\sum_{u, v = -N/2}^{N/2-1} |I(u,v)|^2
= \sum_{(\ox, \oy) \in \Omega_R} w(\ox,\oy) \cdot |\hat{I}(\ox,\oy)|^2
\]
holds if and only if the weights $w(\ox,\oy), (\ox,\oy)\in\Omega_R$ satisfy condition \eqref{eq:weightcondition}.
Moreover, in general, $R$ needs to be at least $16$ for such weights to exist.
\end{theorem}

\subsection{Weight Functions}
\label{subsec:weights}

To avoid high complexity in the computation of the weights satisfying Theorem~\ref{theo:weightedisometry}, we relax
the requirement for exact isometric weighting. Instead of representing the weights as the solution of a large system of
equations, they will be represented in terms of an undercomplete basis for functions on the pseudo-polar grid. More precisely,
we first design basis functions $w_1, \ldots, w_n:\Omega_R\rightarrow \bR^+$ such that $\sum_{j=1}^n w_j(\ox,\oy)\neq 0$ for
all $(\ox,\oy)\in\Omega_R$. We then define the weight function $w:\Omega_R\rightarrow\bR^+$ to be  $w:=\sum_{j=1}^n c_jw_j$,
with $c_1,\ldots,c_n$ being nonnegative constants. These coefficients are determined by solving \eqref{eq:weightcondition}
with respect to this weight function $w$ using the least square method.
We compute the coefficients in this expansion once for a given problem size; then hardwire them in the algorithm.

\subsubsection{Recommended Choices of Weights}
\label{subsec:choice}

In what follows, we present several designs of basis functions, each one providing nearly isometric weighting.
Notice that, slightly abusing notation, we will use $(\ox, \oy)$ and $(k,\ell)$ interchangeably. The weighting
for the three choice we recommend is displayed in Figure \ref{fig:001}.

\underline{\sc Choice 1:}
Our first choice are the following seven functions $w_1, \ldots, w_7$:\\
{\em Center}: $w_1 = 1_{(0,0)} \mbox{ and } w_2=1_{\{(\ox, \oy) : |k|=1\}}$,\\
{\em Boundary}: $w_3=1_{\{(\ox, \oy) : |k|=NR/2,\, \ox=\oy\}} \mbox{ and } w_4=1_{\{(\ox, \oy) : |k|=NR/2,\, \ox\neq\oy\}}$,\\
{\em Seam lines}: $w_5=|k| \cdot 1_{\{(\ox, \oy) : 1<|k|<NR/2,\, \ox=\oy\}}, w_6=1_{\{(\ox, \oy) : |k|=NR/2-3,\,\ox=\oy\}}$,\\
{\em Interior}: $w_7=|k| \cdot 1_{\{(\ox, \oy) : 1<|k|<NR/2,\,\ox\neq\oy\}}$.\\

\underline{\sc Choice 2:}
This is a simplified version of `Choice 1' using the 5 functions:\\
{\em Center}: $w_1 = 1_{(0,0)}$,\\
{\em Boundary}: $w_2=1_{\{(\ox, \oy) : |k|=NR/2,\, \ox=\oy\}} \mbox{ and } w_3=1_{\{(\ox, \oy) : |k|=NR/2,\, \ox\neq\oy\}}$,\\
{\em Seam lines}: $w_4=|k| \cdot 1_{\{(\ox, \oy) : 1\le|k|<NR/2,\, \ox=\oy\}}$,\\
{\em Interior}: $w_5=|k| \cdot 1_{\{(\ox, \oy) : 1\le|k|<NR/2,\,\ox\neq\oy\}}$.\\

\underline{\sc Choice 3:}
Finally, we suggest the following $N/2+2$ functions on the pseudo-polar grid:\\
{\em Center}: $w_1 = 1_{(0,0)}$,\\
{\em Radial Lines}: $w_{\ell+2}=1_{\{(\ox, \oy) :  1<|k|<NR/2,\, \oy=\frac{\ell}{N/2}\ox\}},\quad \ell=0,1,\ldots,N/2$.\\

\begin{figure}[ht]
\begin{center}
\includegraphics[height=1.2in]{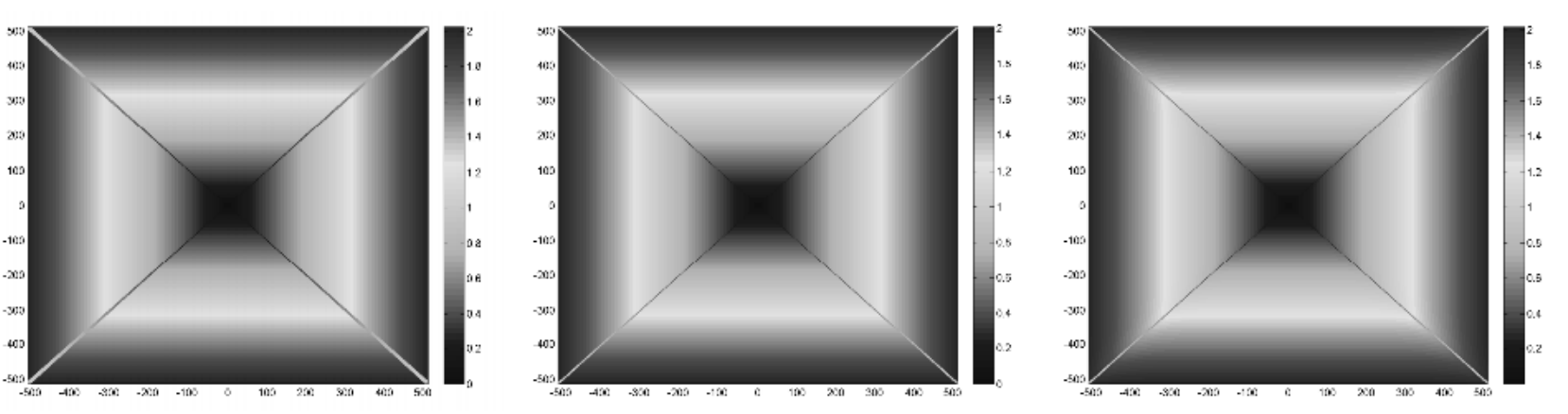}
\end{center}
\caption{Recommended choices of weighting of the pseudo-polar grid for $N=128$ and  $R=8$. The Choice 1 -- 3 are displayed from left to
right.}
\label{fig:001}
\end{figure}

\subsubsection{Comparison of Weight Functions}
\label{subsec:comparison}

The patterns of the weights are seemingly similar in view of Figure~\ref{fig:001}. However, their performances can be quite different
depending on the chosen performance measure. We mention that the measures chosen below are also part of our framework of performance
measures for parabolic scaling algorithms discussed in Section~\ref{sec:qualitymeasures}.

Letting $R=8$, we generate a sequence of 5 random images $I_1$, $\ldots$, $I_5$ of size $N\times N$ with standard normally distributed entries.
We use the following measure to compare the performance of different weights:
\[
M_{isom} := \frac15\sum_{i=1}^5\frac{\|P^\star w P I_i-I_i\|_2}{\|I_i\|_2},
\]
where $P: I\rightarrow \hat{I}$ denotes the PPFT from \eqref{def:ppft} and $w: J\rightarrow J_w$ -- by abusing notation -- denotes the `weighting operator'
$J_w={w} \cdot J$, for a image $J:\Omega_R\rightarrow \bC$ and a weight function
$w:\Omega_R\rightarrow\bR^+$.
Table \ref{tab:1} displays the performance of the weights from Choice 1 -- 3 with respect to this measure.
\begin{table}[ht]
\label{tab:1}
\caption{Comparison using random images}
\begin{tabular}{p{1.7cm}p{1.7cm}p{1.7cm}p{1.7cm}p{1.7cm}p{1.9cm}}
\toprule
$N$ & 32 & 64 & 128 & 256 & 512 \\
\midrule
Choice 1 & 4.3E-3 & 2.6E-3 & 2.2E-3 & 1.4E-3 & 9.3E-4 \\
Choice 2 & 4.2E-3 & 4.0E-3 & 1.8E-3 & 1.5E-3 & 8.8E-4 \\
Choice 3 & 9.8E-3 & 6.2E-3 & 3.4E-3 & 2.1E-3 & N/A \\
\bottomrule
\end{tabular}
\end{table}

Next, we choose the real image `Barbara', which we denote by $I$, and the measure $\frac{\|P^\star w P I-I\|_2}{\|I\|_2}$ to
compare the performance of the different choices of weights.
\begin{table}[ht]
\label{tab:2}
\caption{Comparison using `Barbara'}
\begin{tabular}{p{1.7cm}p{1.7cm}p{1.7cm}p{1.7cm}p{1.7cm}p{1.9cm}}
\toprule
$N$ & 32 & 64 & 128 & 256 & 512 \\
\midrule
Choice 1 & 2.4E-3 & 1.6E-3 & 1.1E-3 & 5.2E-4 & 2.2E-4 \\
Choice 2 & 2.8E-3 & 1.2E-3 & 8.3E-4 & 3.9E-4 & 1.5E-4 \\
Choice 3 & 5.6E-3 & 2.8E-3 & 2.2E-3 & 9.1E-4 & N/A \\
\bottomrule
\hline
\end{tabular}
\end{table}

From  Tables \ref{tab:1} and \ref{tab:2}, it can be seen that the operator $P^\star wP$ seems to converge to an identity operator as
$N \to \infty$. The  data also indicates that it is justifiable to define weights which are linearly
increasing along the radial direction.

When using iterative methods, for instance, the conjugate gradient method, for  computing the inverse, a weight $w$ plays the
role of a preconditioner. Therefore its performance as such can be effectively measured by the condition number of the operator
$P^\star w P$, i.e., $cond(P^\star w P)=\lambda_{max}(P^\star w P)/\lambda_{min}(P^\star w P)$. This measure is
displayed in Table \ref{tab:3} for our selected three choices of weights. Notice that, for each choice of a weight function with the
exception of Choice 3, the condition numbers of $P^\star w P$ are always smaller than $2$.
\begin{table}[ht]
\label{tab:3}
\caption{Comparison of $cond(P^\star w P)$}
\begin{tabular}{p{1.7cm}p{1.7cm}p{1.7cm}p{1.7cm}p{1.7cm}p{1.9cm}}
\toprule
$N$ & 32 & 64 & 128 & 256 & 512 \\
\midrule
Choice 1 &  1.328 & 1.483 & 1.621 &1.726 & 1.834 \\
Choice 2 & 1.379 & 1.503 & 1.621 & 1.731 & 1.833 \\
Choice 3 & 1.760 & 1.887 & 2.001 & 2.104 & N/A \\
\bottomrule
\end{tabular}
\end{table}

\subsection{Windowing}
\label{subsec:Windowing}

According to our discussion of the main steps of the FDST in Section \ref{sec:STfinitedata}, after
performing a weighted pseudo-polar Fourier transform, the data has to be windowed using scaled and sheared versions of a generating
window function.

In this section, we now define such a set of window functions, which we will coin {\em digital shearlets}, and prove that
these -- similar to the continuum domain -- form a tight frame for functions $J : \Omega_R \to \CC$. We remark that this construction
is a digitization of the continuum domain discrete shearlets introduced in \cite{GKL06} (compare also \eqref{eq:shearlets} and
\eqref{eq:shearletsystem}). However, it is far from obvious that again a tight frame is derived, since we here consider a finite
domain.

For the convenience of the reader, we first briefly recall the notion of a tight frame in this particular situation. Let
$f, g : \Omega_R \rightarrow \bC$ be two functions defined on $\Omega_R$. Then, the inner product $\langle f, g\rangle_{\Omega_R}$
is defined to be $\langle f,g\rangle_{\Omega_R}:=\sum_{x\in \Omega_R}(x)\overline{g(x)}$. A sequence
$\{\varphi_\lambda: \Omega_R\rightarrow\bC :\lambda \in \Lambda\}$ with $\Lambda$ being an indexing set is a {\em tight frame} for functions
$J : \Omega_R \to \CC$, if
\[
\sum_{\lambda \in \Lambda} |\langle J,\varphi_\lambda\rangle_{\Omega_R}|^2 = \langle J,J\rangle_{\Omega_R}.
\]
It then follows from basic frame theory (see \cite{Chr03}) that this allows recovery of a function $J : \Omega_R \to \CC$ from
its coefficients $( \langle J,\varphi_\lambda\rangle_{\Omega_R})_{\lambda \in \Lambda}$ by
computing
\[
J = \sum_{\lambda \in \Lambda} \langle J,\varphi_\lambda\rangle_{\Omega_R} \varphi_\lambda.
\]

Despite the danger of repeating ourselves, let us mention that our fundamental goal is to introduce digital shearlets as the
exact digitization of continuum domain shearlets. We now describe the construction step by step, which will give evidence
to the fact that we achieved this goal. There will be one step though -- when defining the modulation -- where we have to
slightly deviate from an exact digitization, and we will explain the reasons for this.

We start by defining the scaling function and the generating digital shearlet. For this, let $j_L:=-\lceil\log_4(R/2)\rceil$,
which will soon be shown to be the lowest possible scale. Let $W_0$ be the Fourier transform of the Meyer scaling function such that
\beq \label{eq:defW0}
\Sp W_0 \subseteq [-1,1]\quad\mbox{and}\quad W_0(\pm 1)=0,
\eeq
and let $V_0$ be a `bump' function satisfying
\[
\Sp V_0 \subseteq [-\tfrac32,\tfrac32] \qquad \mbox{with} \qquad V_0(\xi)\equiv 1\mbox{ for } |\xi|\le 1, \xi\in\bR.
\]
Then we define the {\em scaling function} $\phi$ for the digital shearlet system to be
\[
\hat\phi(\xi_1,\xi_2) = W_0(4^{-j_L}\xi_1)V_0(4^{-j_L}\xi_2), \quad (\xi_1,\xi_2)\in\bR^2.
\]
We will later restrict this function to the pseudo-polar grid.

Let next $W$ be the Fourier transform of the Meyer wavelet function with
\beq \label{eq:supportW}
\Sp W \subseteq [-4,\tfrac14] \cup [\tfrac14,4] \quad \mbox{and} \quad W(\pm \tfrac14)=W(\pm 4)=0,
\eeq
as well as
\beq \label{eq:summabilityW}
|W_0(4^{-j_L}\xi)|^2 + \sum_{j =j_L}^{\lceil\log_4 N\rceil}  |W(4^{-j} \xi)|^2 = 1 \qquad \mbox{for all } |\xi| \le N,\; \xi \in \bR.
\eeq
We further choose $V$ to be a `bump' function  satisfying
\beq \label{eq:supportV}
\Sp V \subseteq [-1,1] \quad \mbox{and} \quad V(\pm 1)=0,
\eeq
and also
\[
|V(\xi-1)|^2 + |V(\xi)|^2 + |V(\xi+1)|^2 = 1 \qquad \mbox{for all } |\xi| \le 1, \; \xi \in \bR.
\]
Notice that this implies
\beq \label{eq:summabilityV2}
\sum_{s=-2^j}^{2^j} |V(2^j\xi-s)|^2 = 1 \qquad \mbox{for all } |\xi| \le 1, \; \xi \in \bR \mbox{ and } j\ge0,
\eeq
which will become important for the analysis of frame properties. For the choice of $V_0$, $W_0$, $V$, and $W$ in our
implementation, we refer to Section~\ref{sec:implementation}.
Then the {\em generating shearlet} $\psi$  is defined as
\beq \label{eq:digitalshearlet}
\hat\psi(\xi_1,\xi_2) = W(\xi_1)V(\tfrac{\xi_2}{\xi_1}), \quad (\xi_1,\xi_2)\in\bR^2.
\eeq

We will now define digital shearlets on $\Omega_R^{21}$ and extend the definition to the other cones by symmetry. At this time, we
assume $R$ and $N$ are both positive, even integers and $N=2^{n_0}$ for some integer $n_0\in\bN$.

For this, we first analyze the exact digitization of the coefficients of the discrete shearlet system from
Subsection \ref{subsec:defST} for a function $J : \Omega_R \to \CC$ by using the shearlets $\psi$ defined in
\eqref{eq:digitalshearlet}. This will lead to the appropriate range of scales and to the support of a scaled and sheared version of
the shearlet $\psi$.

When restricting to the cone $\Omega_R^{21}$, the exact digitization of the coefficients of the discrete
shearlet system is
\beq \label{eeq:digitalshearletcoeff1Step}
\sum_{\omega:=(\ox,\oy) \in \Omega_R^{21}} J(\ox,\oy)
2^{-j\frac{3}{2}} \overline{\hat{\psi}(S_s^T A_{4^{-j}} \omega)} e^{-2\pi i\ip{A_{4^{-j}} S_s m}{\omega}},
\eeq
where $j$, $s$, and $m$ are to be determined.
The choice of $\psi$ leads to the coefficients
\begin{eqnarray*}
\lefteqn{\sum_{\omega:=(\ox,\oy) \in \Omega_R^{21}} J(\ox,\oy)
2^{-j\frac{3}{2}} \overline{W(4^{-j}\omega_x)V(s+2^j\tfrac{\omega_y}{\omega_x})} e^{-2\pi i\ip{A_{4^{-j}} S_s m}{\omega}}}\\
& = & \sum_{k=1}^{RN/2} \sum_{\ell = -N/2}^{N/2} J(\ox,\oy)
2^{-j\frac{3}{2}} \overline{W(4^{-j}\tfrac{2k}{R})}  \overline{V(s-2^{j+1}\tfrac{\ell}{N})} e^{-2\pi i \ip{m}{S_s^T A_{4^{-j}} \omega}}.
\end{eqnarray*}
The support conditions \eqref{eq:supportW} and \eqref{eq:supportV} of $W$ and $V$, respectively, imply
\[
k = 4^{j-1} \tfrac{R}{2} + n_1, \quad n_1 = 0 ,\ldots, 4^{j-1}\cdot \tfrac{15 R}{2},
\]
as well as
\[
\ell = 2^{-j-1} N(s-1) + n_2, \quad n_2 = 0, \ldots, 2^{-j} N,
\]
if we assume $k$ and $\ell$ to be positive integers.

We next analyze the support properties in radial direction.
If $j<-\lceil\log (R/2)\rceil$, then $k<1$, which corresponds to only one point -- the origin --, and this is dealt with by
the scaling function. Hence the lowest possible scale is $j_L=-\lceil\log(R/2)\rceil$.
If $j > \lceil\log_4 N\rceil$, we have $k \ge \frac{RN}{2}$. Hence the value $W(1/4) = 0$ (cf. \eqref{eq:supportW})
is placed on the boundary, and thus these scales can be omitted. This implies that the highest possible scale is $j_H:=\lceil\log_4 N\rceil$.
Hence, the scaling parameter will be chosen to be
\[
j \in  \{j_L, \ldots, j_H\}.
\]
The radial support of the windows associated with scales $j_L < j < j_H$ is
\beq \label{eq:defi_k}
k = 4^{j-1} \tfrac{R}{2} + n_1, \quad n_1 = 0 ,\ldots, 4^{j-1} \cdot \tfrac{15 R}{2},
\eeq
and the radial support of the windows associated with the scales $j_L=-\lceil \log_4(R/2)\rceil$ and $j_H=\lceil\log_4 N\rceil$ is
\beq \label{eq:defi_k_H_L}
\begin{aligned}
k &= n_1,  &n_1&=1,\ldots,4^{j_L+1}\tfrac{R}{2}, &\mbox{ for } &j=j_L,\\
k &= 4^{j_H-1} \tfrac{R}{2}+n_1, &n_1&=0, \ldots, \tfrac{RN}{2}-4^{j_H-1}\tfrac{R}{2},& \mbox{ for }
&j = j_H.
\end{aligned}
\eeq

We further analyze the precise support properties in angular direction. First, we examine the case $j\ge 0$. If $s > 2^j$, we
have $\ell \ge N/2$. Hence the value $V(-1)=0$ (cf. \eqref{eq:supportV}) is placed on the seam line, and these parameters can be
omitted. By symmetry, we also obtain $s \ge -2^j$. Thus the shearing parameter will be chosen to be
\[
s \in \{-2^j,\ldots,2^j\}
\]
The angular support of the windows at scale $j$ associated with shears $-2^j < s < 2^j$ is
\beq \label{eq:defi_ell}
\ell = 2^{-j-1} N(s-1) + n_2, \quad n_2 = 0, \ldots, 2^{-j} N,
\eeq
the angular support at scale $j$ associated with the shear parameters $s_L=-2^j$ and $s_H=2^j$ is
\begin{equation}\label{eq:defi_ell_U}
\begin{aligned}
\ell &= 2^{-j-1} N(s_L-1) + n_2, \quad n_2 = 2^{-j} \tfrac{N}{2}, \ldots, 2^{-j} N, &\mbox{for }s&=s_L,
\\\ell &= 2^{-j-1} N(s_H-1) + n_2, \quad n_2 = 0, \ldots, 2^{-j} \tfrac{N}{2},& \mbox{for }s&=s_H.
\end{aligned}
\end{equation}
For the case  $j<0$, we simply let $s=0$ and $\ell=-N/2+n_2$ with $n_2=0,\ldots,N$. Also, in this case, the window
function $W(4^{-j}\ox)V(s+2^j\tfrac{\oy}{\ox})$ is slightly modifed to be $W(4^{-j}\ox)V_0(s+2^j\tfrac{\oy}{\ox})$ so that
the tight frame property still holds.

These computations allow us to determine the support size of $W(4^{-j}\omega_x) V(s+2^j\frac{\omega_y}{\omega_x})$ in terms
of pairs $(k,\ell)$. In fact, the number of sampling points in radial and angular direction affected by a window at scale
$j$ and shear $s$ are
\beq \label{eq:L1}
\cL^1_{j} =\left\{ \begin{array}{cll}
4^{j+1}\tfrac{R}{2} &:& j= j_L,\\[0.5ex]
4^{j-1} \cdot \frac{15 R}{2} + 1 & : & j_L<j<j_H,\\[0.5ex]
\tfrac{RN}{2}-4^{j-1}\tfrac{R}{2}+1 & : & j=j_H,
 \end{array} \right.
\eeq
and
\beq \label{eq:L2}
\cL^2_{j,s} =\left\{ \begin{array}{cll}
2^{-j} N + 1 & : & -2^j < s < 2^j\;\mbox{ with }\;j\ge0,\\
2^{-j} \frac{N}{2}+1 & : & s \in \{-2^j,2^j\}\;\mbox{ with }\; j\ge0,\\
N+1 & : & j<0,
 \end{array} \right.
\eeq
respectively.

Notice that the support of the continuum function $\xi \mapsto W(4^{-j}\xi_1)  V(s+2^j\frac{\xi_2}{\xi_1})$
is of approximate size $4^j \times 2^j$ obeying parabolic scaling. The situation is however different in the
digital realm. Since the sampling density in angular direction does change with growing radius -- the sampling
grid becomes in fact coarser --, parabolic scaling is not such directly mirrored in the relation between
$\cL^1_j$ and $\cL^2_{j,s}$.

Let us next carefully examine the exponential term, which can be written as
\[
e^{-2\pi i \ip{m}{S_s^T A_{4^{-j}} \omega}}
= e^{-2\pi i \ip{m}{(4^{-j}\omega_x,4^{-j}s\omega_x + 2^{-j} \omega_y)}}
= e^{-2\pi i \ip{m}{(4^{-j}\frac{2k}{R},4^{-j}s\frac{2k}{R} - 2^{-j} \frac{4\ell k}{RN})}}.
\]
We now adjust the exponential term as illustrated in Figure \ref{fig:adjustmentexp}, which
will be the only slight adaption we allow us to make when digitizing. The reason is to enable a
direct application of the inverse fast Fourier transform. The necessity for this modification
occurs because of two reasons:
\bitem
\item[1.] We cannot make the change of variables $\tau := S_s^T A_{4^{-j}} \omega$ in formula
\eqref{eeq:digitalshearletcoeff1Step}, which is the first step in the `continuous' proof for tightness,
due to the fact that the pseudo-polar grid is {\em not} invariant under the action of $S_s^T A_{4^{-j}}$.
\item[2.] The Fourier transform of a function defined on the pseudo-polar grid does {\em not} satisfy any
Plancherel equation.
\eitem
\begin{figure}[ht]
\begin{center}
\includegraphics[height=1.05in]{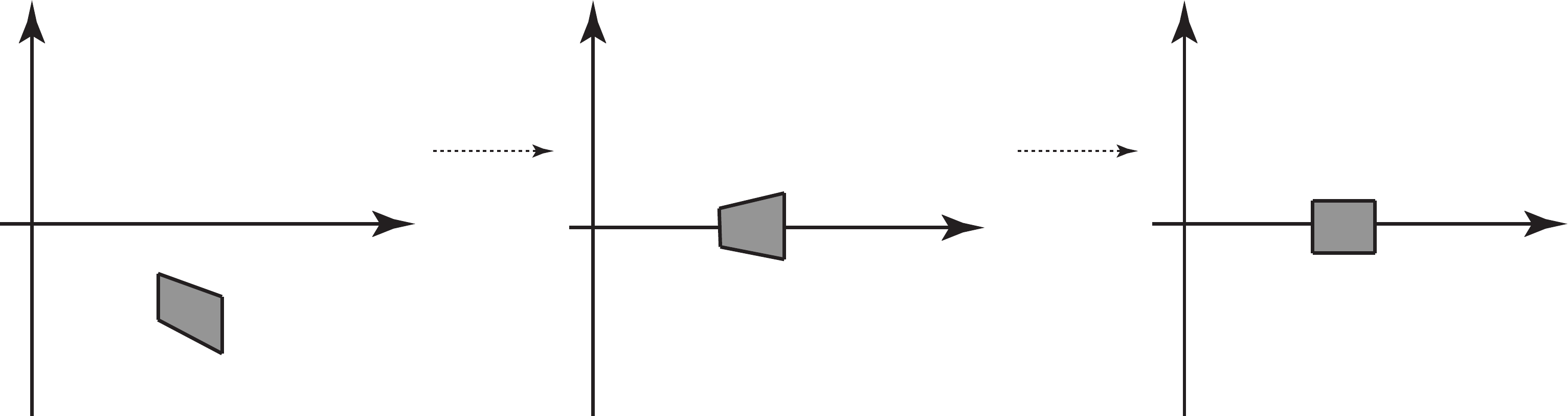}
\put(-212,52){$(S_s^T)^{-1}$}
\put(-90,52){$\theta$}
\end{center}
\caption{Adjustment of the exponential term through the map $\theta \circ (S_s^T)^{-1}$.}
\label{fig:adjustmentexp}
\end{figure}
Defining $\theta : \bR\setminus\{0\} \to \bR$ by $\theta(x,y) = (x,\tfrac{y}{x})$, we let the new exponential term
be
\[
e^{-2\pi i \ip{m}{(\theta \circ (S_s^T)^{-1})(4^{-j}\frac{2k}{R},4^{-j}s\frac{2k}{R} - 2^{-j} \frac{4\ell k}{RN})}}
= e^{-2\pi i \ip{m}{(4^{-j}\frac{2k}{R},-2^{j+1}\frac{\ell}{N})}}.
\]
This exponential term can be rewritten as
\[
e^{-2\pi i \ip{m}{(4^{-j}\frac{2k}{R},-2^{j+1}\frac{\ell}{N})}}
= e^{-2\pi i (\frac{m_1}{4}+(1-s)m_2)} e^{-2\pi  i \ip{m}{(4^{-j}\frac{2n_1}{R},-2^{j+1}\frac{n_2}{N})}},
\]
with $n_1$ and $n_2$ ranging over an appropriate set defined by \eqref{eq:defi_k}, \eqref{eq:defi_k_H_L},
and \eqref{eq:defi_ell}--\eqref{eq:defi_ell_U}. The reformulation -- recall the definitions of $\cL^1_j$ and $\cL^2_{j,s}$ in
\eqref{eq:L1} and \eqref{eq:L2}~--
\[
\exp\left(-2\pi i \ip{m}{\left(\tfrac{\cL^1_j 4^{-j}(2/R)}{\cL^1_j}n_1,\tfrac{-\cL^2_{j,s}2^{j+1}(1/N)}{\cL^2_{j,s}}n_2\right)}\right), \quad n_1, \; n_2,
\]
shows that we might regard the exponential terms as characters of a suitable locally compact abelian group (see \cite{HR63})
with annihilator identified with the rectangle
\[
\cR_{j,s} = \{((\cL^1_j)^{-1} 4^{j}\tfrac{R}{2} \cdot r_1, -(\cL^2_{j,s})^{-1}\tfrac{N}{2^{j+1}} \cdot r_2) : r_1 = 0 ,\ldots, \cL^1_j-1,\:
r_2 = 0, \ldots, \cL^2_{j,s} -1\}.
\]
For the low frequency square we further require the set
\[
\cR = \{(r_1, r_2) : r_1 = -1 ,\ldots, 1,\: r_2 = -\tfrac{N}{2}, \ldots, \tfrac{N}{2}\}.
\]
We are now ready to define digital shearlets defined on the pseudo-polar grid $\Omega_R$.

\begin{definition} \label{defi:digitalshearlets}
Retaining the definitions and notations derived in this subsection and Subsection \ref{subsec:wPPFT}, on the cone $\Omega_R^{21}$
we define {\em digital shearlets} at scale $j \in  \{j_L, \ldots, j_H\}$, shear $s \in \{-2^j, \cdots, 2^j\}$, and
 spatial position $m \in \cR_{j,s}$ by
\begin{eqnarray*}
\sigma_{j,s,m}^{21}(\ox, \oy)
 =   \tfrac{C(\ox, \oy) }{\sqrt{|\cR_{j,s}|}} \, W(4^{-j} \ox) \, V^j(s+2^{j}\tfrac{\oy}{\ox})\chi_{\Omega_R^{21}}(\ox, \oy)\,
e^{2\pi i \ip{m}{(4^{-j}\ox,2^{j}\tfrac{\oy}{\ox})}},
\end{eqnarray*}
where $V^j = V$ for $j\ge0$ and $V^j=V_0$ for $j<0$, and
\[
C(\ox, \oy) = \left\{\begin{array}{cll}
1 & : &  (\ox, \oy) \not\in \cS_R^1 \cup \cS_R^2,\\[0.5ex]
\frac{1}{\sqrt{2}} & : & (\ox, \oy) \in (\cS_R^1 \cup \cS_R^2)\setminus \cC,\\[0.5ex]
\frac{1}{\sqrt{2(N+1)}} & : & (\ox, \oy) \in \cC.
\end{array}\right.
\]
The shearlets $\sigma_{j,s,m}^{11}, \sigma_{j,s,m}^{12}, \sigma_{j,s,m}^{22}$ on the
remaining cones are defined accordingly by symmetry with equal indexing sets for scale $j$, shear $s$, and spatial location $m$.
For $\iota_0=1,2$ and $n \in \cR$, we further define the functions
\[
\varphi_n^{\iota_0}(\ox, \oy)  =
\tfrac{C(\ox, \oy)}{\sqrt{|\cR|}} \hat{\phi}(\ox, \oy) \chi_{\Omega_R^{\iota_0}}(\ox, \oy)\, e^{2\pi i\ip{n}{(\frac{k}{3},\frac{\ell}{N+1})}}.
\]
Summarizing, we call the system
\begin{eqnarray*}
\cD\cS\cH& \hspace*{-0.05cm} = \hspace*{-0.05cm}& \{\varphi_n^{\iota_0} : \iota_0=1,2, n \in \cR\}
\cup\{\sigma_{j,s,m}^{\iota}: j \in \{j_L, \ldots, j_H\}, s\in \{-2^j, \cdots, 2^j\}, \\& &\hspace*{7cm}  m \in \cR_{j,s},\iota=11,12,21,22\}
\end{eqnarray*}
the {\em digital shearlet system}.
\end{definition}

The just defined digital shearlet system forms indeed a tight frame for functions $J : \Omega_R \to \CC$, as the following
results shows.

\begin{theorem}\label{thm:DSHtight}
The digital shearlet system $\cD\cS\cH$ defined in Definition \ref{defi:digitalshearlets}
forms a tight frame for functions $J : \Omega_R \to \CC$.
\end{theorem}

\begin{proof}
Letting $J : \Omega_R \to \CC$, we claim that
\begin{equation}\label{eq:RHSLHS}
\langle J,J\rangle_{\Omega_R}=\sum_{\iota_0,n}|\langle J,\varphi_n^{\iota_0} \rangle_{\Omega_R}|^2
+\sum_{\iota,j,s,m}|\langle J,\sigma_{j,s,m}^{\iota} \rangle_{\Omega_R}|^2
\end{equation}
which proves the result.

We start by analyzing the first term on the RHS of \eqref{eq:RHSLHS}. Let $\iota_0 \in \{1, 2\}$ and
$J_C:\Omega_R\rightarrow\bC$ be defined by $J_C(\ox,\oy):=C(\ox,\oy)\cdot J(\ox,\oy)$ for $(\ox,\oy)\in\Omega_R$.
Using the support conditions of $\hat\phi$,
\begin{eqnarray} \nonumber
\hspace*{-0.5cm} &&\sum_n|\langle J,\varphi_n^{\iota_0}\rangle_{\Omega_R}|^2=\sum_n\Big|\sum_{(\ox,\oy) \in \Omega_R^{\iota_0}} J(\ox,\oy) \overline{\varphi_n^{\iota_0}(\ox,\oy)}\Big|^2\\  \nonumber
& = & \frac{1}{|\cR|} \sum_n \Big|\sum_{(\ox,\oy) \in \Omega_R^{\iota_0}}   J_C(\ox,\oy)\cdot
\hat{\phi}(\ox, \oy) \cdot e^{-2\pi i\ip{n}{(\frac{k}{3},\frac{\ell}{N+1})}}\Big|^2\\
\label{eq:last1}
& = & \frac{1}{|\cR|} \sum_n \Big|\sum_{k=-1}^{1} \sum_{\ell=-N/2}^{N/2}
 J_C(\ox,\oy)\cdot \hat{\phi}(\ox, \oy) \cdot e^{-2\pi i\ip{n}{(\frac{k}{3},\frac{\ell}{N+1})}}\Big|^2.
\end{eqnarray}
The choice of $\cR$ now allows us to use the Plancherel formula. Exploiting again support properties, we conclude from
\eqref{eq:last1} that
\[
\sum_n|\langle J,\varphi_n^{\iota_0}\rangle_{\Omega_R}|^2=\sum_{(\ox,\oy) \in \Omega_R^{\iota_0}} |C(\ox,\oy) \cdot J(\ox,\oy)|^2 \cdot |\hat{\phi}(\ox, \oy)|^2.
\]
Combining $\iota_0=1,2$ and using \eqref{eq:defW0}, we proved
\beq \label{eq:firstterm}
\sum_{\iota_0}\sum_n|\langle J,\varphi_n^{\iota_0}\rangle_{\Omega_R}|^2=
\sum_{(\ox,\oy) \in \Omega_R} |J(\ox,\oy)|^2 \cdot |W_0(\ox)|^2.
\eeq

Next we study the second term on the RHS in \eqref{eq:RHSLHS}. By symmetry, it suffices to consider the case $\iota=21$.
By the support conditions on $W$ and $V$ (see \eqref{eq:supportW} and \eqref{eq:supportV}),
{\allowdisplaybreaks
\begin{eqnarray} \nonumber
&&\sum_{j,s,m}|\langle J,\sigma_{j,s,m}^{21}\rangle_{\Omega_R}|^2=
\sum_{j,s} \sum_{m \in \cR_{j,s}} \Big|\sum_{(\ox,\oy) \in \Omega_R^{21}} J(\ox,\oy)
\overline{\sigma_{j,s,m}^{21}(\ox,\oy)}\Big|^2\\  \nonumber
& = & \sum_{j,s} \frac{1}{|\cR_{j,s}|} \sum_{m \in \cR_{j,s}} \Big|\sum_{(\ox,\oy) \in \Omega_R^{21}} J_C(\ox,\oy) \cdot \overline{W(4^{-j} \ox)}
\\  \nonumber & & \cdot
 \overline{V^j(s+2^{j}\tfrac{\oy}{\ox})} \cdot e^{-2\pi i\ip{m}{(4^{-j}\ox,2^{j}\frac{\oy}{\ox})}}\Big|^2\\  \nonumber
& = & \sum_{j,s} \frac{1}{|\cR_{j,s}|} \sum_{m \in \cR_{j,s}} \Big|\sum_{k=4^{j-1} (R/2)}^{4^{j+1} (R/2)} \sum_{\ell = 2^{-j-1}N(s-1)}^{2^{-j-1}N(s+1)}
 J_C(\ox,\oy)  \\ \label{eq:last2}
& &  \cdot  \overline{W(4^{-j} \ox)}\cdot \overline{V^j(s+2^{j}\tfrac{\oy}{\ox})} \cdot e^{-2\pi i\ip{m}{(4^{-j}\frac{2k}{R},-2^{j+1}\frac{\ell}{N})}}\Big|^2.
\end{eqnarray}
}
Similarly as before, the choice of $\cR_{j,s}$ does allow us to use the Plancherel formula. Hence \eqref{eq:last2} equals
\[
\sum_{j,s,m}|\langle J,\sigma_{j,s,m}^{21}\rangle_{\Omega_R}|^2=
\sum_{j,s} \sum_{(\ox,\oy) \in \Omega_R^{21}} \Big| J_C(\ox,\oy)\cdot\overline{W(4^{-j} \ox)V^j(s+2^{j}\tfrac{\oy}{\ox})}\Big|^2.
\]
Next we use \eqref{eq:summabilityV2} to obtain
\begin{eqnarray} \nonumber
\lefteqn{\sum_{j,s} \sum_{(\ox,\oy) \in \Omega_R^{21}} \Big| J_C(\ox,\oy) \cdot\overline{W(4^{-j} \ox)} \cdot\overline{V^j(s+2^{j}\tfrac{\oy}{\ox})}\Big|^2}\\ \nonumber
& = & \hspace*{-0.25cm} \sum_{(\ox,\oy) \in \Omega_R^{21}} | J_C(\ox,\oy)|^2 \sum_{j =j_L}^{j_H}
 |W(4^{-j} \ox)|^2 \cdot  \sum_{s=-2^j}^{2^j} |V^j(s+2^{j}\tfrac{\oy}{\ox})|^2 \\ \nonumber
& = & \hspace*{-0.25cm} \sum_{(\ox,\oy) \in \Omega_R^{21}} | J_C(\ox,\oy)|^2 \sum_{j = j_L}^{j_H}
 |W(4^{-j} \ox)|^2.
\end{eqnarray}
Hence the second term on the RHS in \eqref{eq:RHSLHS} equals
\beq \label{eq:secondterm}
\sum_{\iota}\sum_{j,s,m}|\langle J,\sigma_{j,s,m}^{\iota}\rangle_{\Omega_R}|^2=\sum_{(\ox,\oy) \in \Omega_R} |J(\ox,\oy)|^2\cdot \sum_{j =j_L}^{j_H}.
 |W(4^{-j} \ox)|^2.
\eeq
Finally, our claim \eqref{eq:RHSLHS} follows from combining \eqref{eq:firstterm}, \eqref{eq:secondterm}, and \eqref{eq:summabilityW}.
\end{proof}


\section{Inverse FDST for Finite Data}
\label{sec:ISTfinitedata}

In this section, we will analyze the inverse of the FDST. For this, let $P$ and $w$ denote
the operators for the pseudo-polar Fourier transform and the weighting defined as before. By slight abuse of notation,
we will further let $W$ denote the windowing with respect to the digital shearlets, i.e.,
for $J:\Omega_R\rightarrow\bC$,
\[
WJ:=\{\langle J,\varphi_n^{\iota_0}\rangle_{\Omega_R} : \iota_0,n\}\cup\{\langle J,\sigma_{j,s,m}^{\iota}\rangle_{\Omega_R} : \iota, j,s,m\}
\]
with $\varphi^{\iota_0}_n$ and $\sigma_{j,s,m}^\iota$ being the digital shearlets (see Definition~\ref{defi:digitalshearlets}) for
$j=j_L$, $\ldots$, $j_H$, $s \in \{-2^j, \cdots, 2^j\}$, and $n\in\cR, m\in \cR_{j,s}$. Then the FDST, which we abbreviate by $S$, takes the form
\beq \label{eq:formS}
S = W\sqrt{w}P.
\eeq
To be more precise, letting $I$ be an image of size $N\times N$ and $J_w:=\sqrt{w}PI$ be the weighted pseudo-polar Fourier transform of $I$,
the set of shearlet coefficients of $I$ generated by the FDST can be written as
\begin{equation}\label{def:SHX}
SI = \{c^{\iota_0}_n: \iota_0, n\}\cup\{c^{\iota}_{j,s,m} : \iota, j, s, m\},
\end{equation}
where $c^{\iota_0}_n = \langle J_w, \varphi^{\iota_0}_n\rangle_{\Omega_R}$ for $\iota_0=0,1$ and
$c^{\iota}_{j,s,m} = \langle J_w, \sigma^{\iota}_{j,s,m}\rangle_{\Omega_R}$ for $\iota=11,12,21,22$.

Aiming to derive a closed form of $S^{-1}$, let $P^\star$ denote the adjoint operator of $P$, which, for a function $J:\Omega_R\rightarrow \bC$,
is given by
\[
P^\star J(u,v) = \sum_{(\ox,\oy)\in\Omega_R}J(\ox,\oy)e^{\frac{2\pi i}{m_0}(u\ox+v\oy)},\quad u,v=-\tfrac{N}{2},\ldots,\tfrac{N}{2}-1.
\]
Further, let $W^\star$ denote the adjoint operator of $W$, which, given a sequence of shearlet coefficients $C_I=SI$ as in \eqref{def:SHX},
is defined by
\[
W^\star C_I=\sum_{\iota_0,  n\in \cR}c_n^{\iota_0}\varphi_n^{\iota_0}
+ \sum_{\iota, j,s,m}c_{j,s,m}^\iota\sigma_{j,s,m}^\iota.
\]
Then we have the following result, which shows that the inverse FDST equals the adjoint FDST
provided the weight function $w$ satisfies \eqref{eq:weightcondition}.

\begin{proposition}
If  $w$ satisfies \eqref{eq:weightcondition}, then $S^{-1} = P^\star\sqrt{w}W^\star$.
\end{proposition}

\begin{proof}
By hypothesis on $w$, we have $P^\star wP =  {\Id}$. Moreover, $W^\star W = {\Id}$ by Theorem~\ref{thm:DSHtight}.
Since $S = W\sqrt{w}P$ by \eqref{eq:formS}, the result is proved.
\end{proof}

Let us now assume that the weight function $w$ does not fulfill the requirements in \eqref{eq:weightcondition}, i.e.,
$P^\star wP \neq {\Id}$, which prevents simply using the adjoint operator. In this case, we mention two methods to
compute the inverse. The first method is a direct approach by resampling trigonometric polynomials from the pseudo-polar
grid to the Cartesian grid, which is discussed in detail in \cite{ACDIS08}. The second method is an iterative approach
by using the conjugate gradient method, which we now describe.

Suppose we are given a sequence of shearlet coefficients $C_I$ of some image $I$. By Theorem~\ref{thm:DSHtight}, the
inverse digital shearlet windowing of $I$ is $\sqrt{w}PI = W^\star C_I =: J_w$. Consequently, the original image $I$
can be computed by solving
\[
\min_{I\in \bR^{N\times N}} \|\sqrt{w}P I -J_w\|_2.
\]
Solving this problem is equivalent to solving the linear system of equations
\beq \label{eq:CGProblem}
P^\star w P I = P^\star \sqrt{w} J_w.
\eeq
This shows that $w$ plays the role of a preconditioner for the following normal equations:
\[
P^\star P I = P^\star  J_w.
\]
We refer to Subsection~\ref{subsec:weights}
for a selection of choices for weight functions
and a discussion about their performance as preconditioners. Since the matrix corresponding to $P^\star w P$ is symmetric
and positive definite, the conjugate gradient method can be used to solve \eqref{eq:CGProblem}. Algorithmic details
will be discussed in Section~\ref{sec:implementation}. Let us just mention two main issues: The number of iterations
required by the conjugate gradient method depends on the condition number of $P^\star wP$.
Moreover, the conjugate gradient method only requires
applications of $P^\star$ and $P$ to a vector, which can be computed $O(N^2\log N)$ flops, in contrast to forming the
complete matrices.


\section{Mathematical Properties of the FDST}
\label{sec:props}

Results on decay properties of the discrete shearlet coefficients and even more its sparse approximation properties are
well-known, see \cite{GL07a,DK08b,KL10}. These continuum domain results do however not directly imply similar statements
for the introduced digital shearlets due to the fundamentally different nature of a digital grid. Therefore, in this
section, we will analyze decay properties of the digital shearlet coefficients for linear singularities. Moreover, we
will prove shear invariance of the FDST.

\subsection{Decay Properties of FDST Coefficients for Linear Singularities}
\label{subsec:line}

One main advantage of shearlets over wavelets is their ability to precisely resolve curvilinear -- hence, in particular, linear --
singularities due to their anisotropic shape and their additional direction-sensitive shear parameter \cite{KL09}. Our computations
will show that this property carries over to the digital setting. Our model for a linear singularity will be a line. We remark
that similar results can be shown for a Heaviside functions as a model.

Let $I$ be an image of size $N\times N$ with an edge through the origin of slope $t$ satisfying $|t|<1$, i.e.,
$I(u,v) = \delta({tu-v})$, $-N/2\le u,v\le N/2-1$.
\begin{figure}[ht]
\begin{center}
\includegraphics[height=1in]{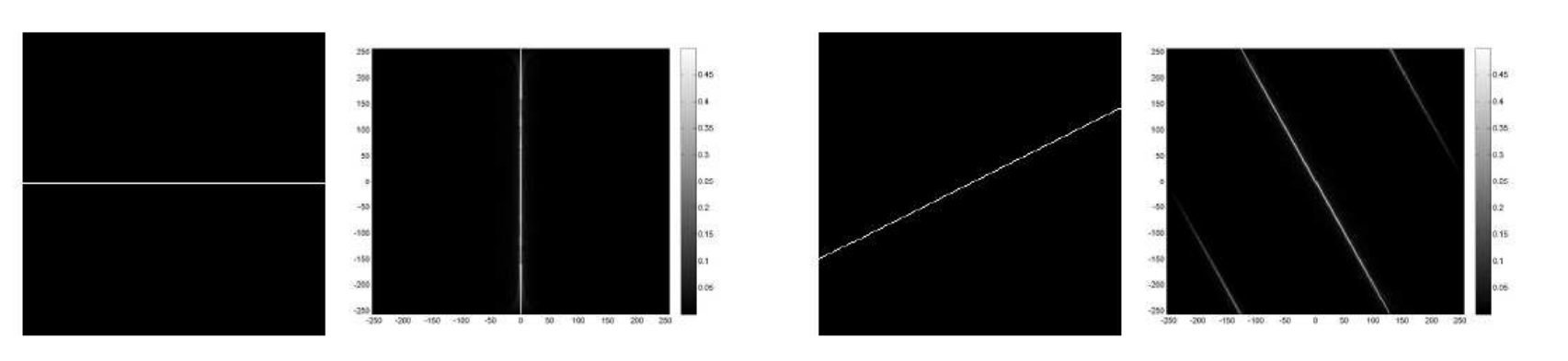}
\end{center}
\caption{From left to right: A horizontal line in spatial domain and its pseudo-polar Fourier transform, and
a line with slope $< 1$ in spatial domain  and its pseudo-polar Fourier transform.}
\label{fig:Line}
\end{figure}
Figure~\ref{fig:Line} indicates that for $t \neq 0$ an aliasing effect occurs. This effect is quite small in comparison
with the intensity of the line, wherefore we will ignore it in the calculations. Also, by symmetry, we can assume $t\ge0$ without loss of generality.
Setting $m_0 = \frac{2}{R}(RN+1)$, for  $(\ox,\oy) = (\frac{2k}{R},-\tfrac{2k}{R}\cdot\frac{2\ell}{N})\in\Omega_R^2$, we have
\[
\hat I(\ox,\oy)
= \sum_{u=-N/2}^{N/2-1} I(u,tu)e^{-\frac{2\pi i}{m_0}(u\ox+tu\oy)}
=\sum_{u=-N/2}^{N/2-1}e^{-\frac{2\pi i}{RN+1}u\cdot k\cdot(1-\frac{2\ell}{N}t)}.
\]
If $k=0$, we obtain $\hat I(\ox,\oy) = N$. If $k \neq 0$,
\[
\hat I(\ox,\oy) = e^{\tfrac{\pi i k}{RN+1}(1-\frac{2\ell}{N}t)} \cdot\frac{\sin(\pi N k(1-\frac{2\ell}{N}t)/(RN+1))}{\sin(\pi k(1-\frac{2\ell}{N}t)/(RN+1))}.
\]
Concluding,
\[
\begin{aligned}
\hat I(\ox,\oy) &= N, &\mbox{ for } k&=0\\
|\hat I(\ox,\oy)|&\le \frac{RN+1}{2|k||1-\frac{2\ell}{N}t|}, &\mbox{ for } |k| &= 1, \ldots, \tfrac{RN}{2}.
\end{aligned}
\]
Similarly, for $(\ox,\oy) = (-\tfrac{2k}{R}\cdot\frac{2\ell}{N},\frac{2k}{R})\in\Omega_R^1$,
\[
\begin{aligned}
\hat I(\ox,\oy) &= N,& \mbox{ for } \ell&=t\cdot \tfrac{N}{2}\\
\qquad |\hat I(\ox,\oy)|&\le \frac{(RN+1)N}{4|k||\ell-\frac{N}{2}t|},& \mbox{ for }|k| &= 1, \ldots, \tfrac{RN}{2}, \: \ell \neq t\cdot \tfrac{N}{2}.
\end{aligned}
\]

We now distinguish two cases. On $\Omega_R^{21}$ (similarly on $\Omega_R^{22}$), for those digital shearlets {\em not} on the seam
lines, i.e., $s \not\in \{-2^j,2^j\}$, we obtain
{\allowdisplaybreaks
\begin{eqnarray}\nonumber
|\langle\hat{I},\sigma^{21}_{j,s,m}\rangle_{\Omega_R}|
&\le&\frac{1}{\sqrt{|\cR_{j,s}|}}\sum_{(\ox,\oy)\in\Sp\sigma_{j,s,0}^{21}}|\hat{I}(\ox,\oy)|\\ \nonumber
&\le&\frac{1}{\sqrt{|\cR_{j,s}|}}\sum_{\ell = 2^{-j-1}N(s-1)}^{ 2^{-j-1}N(s+1)}\frac{N/2}{|N/2-t\ell|}\sum_{k=4^{j-1}R/2}^{4^{j+1}R/2}\frac{RN+1}{2|k|}\\ \nonumber
&=&\frac{(RN+1)N}{4\sqrt{|\cR_{j,s}|}}\sum_{\ell = 2^{-j-1}N(s-1)}^{ 2^{-j-1}N(s+1)}\frac{1}{|N/2-t\ell |}\sum_{k=4^{j-1}R/2}^{4^{j+1}R/2}\frac{1}{|k|}\\ \nonumber
&\le&\frac{(RN+1)N}{4\sqrt{|\cR_{j,s}|}}\sum_{\ell = 2^{-j-1}N(s-1)}^{ 2^{-j-1}N(s+1)}\frac{1}{N/2(1-2^{-j}(s+1)t)} \sum_{k=4^{j-1}R/2}^{4^{j+1}R/2}\frac{1}{|k|}\\ \label{eq:decay1}
&\le&\frac{(RN+1)N}{4\sqrt{|\cR_{j,s}|}}\cdot\frac{1}{2^j(1-t)}\cdot 4\log(2).
\end{eqnarray}
}
On $\Omega_R^{11}$ (similarly on $\Omega_R^{12}$),
{\allowdisplaybreaks
\begin{eqnarray}\nonumber
|\langle\hat{I},\sigma^{11}_{j,s,m}\rangle_{\Omega_R}|
&\le&\frac{1}{\sqrt{|\cR_{j,s}|}}\sum_{(\ox,\oy)\in\Sp\sigma_{j,s,0}^{11}}|\hat{I}(\ox,\oy)|\\ \nonumber
&\le&\frac{1}{\sqrt{|\cR_{j,s}|}}\sum_{\ell = 2^{-j-1}N(s-1)}^{ 2^{-j-1}N(s+1)}\sum_{k=4^{j-1}R/2}^{4^{j+1}R/2}
\Big(\delta(s-t\cdot \tfrac N2)\delta(\ell-t\cdot \tfrac N2) N\\ \nonumber
&&\hspace*{2cm}+(1-\delta(s-t\cdot \tfrac N2)\delta(\ell-t\cdot N/2))\frac{(RN+1)N}{4|k||\ell-t\cdot \tfrac N2|}\Big)\\ \nonumber
&\le& \frac{(RN+1)N}{4\sqrt{|\cR_{j,s}|}}
\sum_{\ell = 2^{-j-1}N(s-1),\ell\neq t\cdot \tfrac N2}^{ 2^{-j-1}N(s+1)}\frac{1}{|\ell-t\cdot N/2|}\sum_{k=4^{j-1}R/2}^{4^{j+1}\cdot \tfrac{R}{2}+1}\frac{1}{|k|}\\ \nonumber
&&\hspace*{3.5cm}+\frac{\delta(s-t\cdot \tfrac N2)\cdot (4^{j+1}\cdot \tfrac{R}{2}+1)N}{\sqrt{|\cR_{j,s}|}}\\ \label{eq:decay2}
&\le&\frac{(RN+1)N\log(N/2)4\log(2)}{4\sqrt{|\cR_{j,s}|}}+\frac{\delta(s-t\cdot \tfrac N2)(4^{j+1}\cdot \tfrac R2+1)N}{\sqrt{|\cR_{j,s}|}}.
\end{eqnarray}
}
Digital shearlets on the seam lines can be dealt with similarly. Thus, by \eqref{eq:decay1} and \eqref{eq:decay2}, we have
upper bounds for shearlet coefficients which as $j \to \infty$ have asymptotic decay
\bitem
\item $O(2^{-3j/2})$, if the shearlet is in $\Omega_R^{2 \cdot }$ and is not aligned with the line singularity,
\item $O(2^{-j/2})$, if the shearlet is in $\Omega_R^{1 \cdot }$ and is not aligned with the line singularity,
\item $O(2^{3j/2})$, if the shearlet is in $\Omega_R^{1 \cdot }$ and is aligned with the line singularity.
\eitem

\bigskip

We now aim to show that line singularities can indeed be detected in the transform domain. For this, we will show a
lower bound on the decay of the shearlet coefficients if the shearlet is aligned with the line singularity.
To avoid unnecessary technicalities, we now assume $\hat{I}(\ox,\oy) = N$
for $\oy/\ox=-1/t$ and $\hat{I}(\ox,\oy)=0$ for $\oy/\ox\neq -1/t$ with $0<t<1$. The case for $t\le0$ can be handled similarly.
Notice that this condition merely assumes that the aliasing effects are negligible for the decay analysis. We further assume
that the window function $W$ and the bump function $V$ are positive on their supports. On $\Omega_R^{11}$,
{\allowdisplaybreaks
\begin{eqnarray*}
\max_{j,s,m}\{|\langle \hat{I}, \sigma^{11}_{j,s,m}\rangle_{\Omega_R} |\}
&\ge& |\langle \hat{I}, \sigma^{11}_{j,s,0}\rangle_{\Omega_R} |\\
&=&\left| \frac{1}{{\sqrt{|\cR_{j,s}|}}} \sum_{(\ox,\oy)\in\Omega_R^{11}}\hat{I}(\ox,\oy)\overline{W(4^{-j}\ox)V^j(s+2^j\tfrac{\oy}{\ox})}\right|\\
&=&\left| \frac{1}{{\sqrt{|\cR_{j,s}|}}} \sum_{(\ox,\oy)\in\Omega_R^{21}}\delta(s-t\cdot\tfrac{N}{2})\cdot N\cdot\overline{W(4^{-j}\ox)V^j(s+2^j\tfrac{\oy}{\ox})}\right|\\
&=&\frac{\delta(s-t\cdot\tfrac{N}{2})\cdot N}{{\sqrt{|\cR_{j,s}|}}} S_{j,s},
\end{eqnarray*}
}
where $S_{j,s}=|\sum_{(\ox,\oy)\in\Omega_R}{W(4^{-j}\ox)V(s+2^j\frac{\oy}{\ox})}|$ denotes the area of the window function with respect to $j$ and $s$.
Since $S_{j,s}$ is approximately the same order as $|\cR_{j,s}|$, we conclude that
\[
\max_{j,m}\{|\langle \hat{I}, \sigma^{11}_{j,tN/2,m}\rangle_{\Omega_R} |\} \ge \frac{N\cdot S_{j,s}}{{\sqrt{|\cR_{j,s}|}}}
=O(2^{j/2}) \quad \mbox{as } j \to \infty.
\]
This estimate combined with \eqref{eq:decay2} shows that the asymptotic decay as $j \to \infty$ in $\Omega_R^{1 \cdot }$ is
\bitem
\item $O(2^{-j/2})$, if the shearlet is not aligned with the line singularity,
\item $\Omega(2^{j/2})$, if the shearlet is aligned with the line singularity.
\eitem

In this sense there is a strong difference between the decay rates of shearlet coefficients  between those aligned with the line singularity and
those {\em not} aligned with  the line singularity.

Experimental results strongly support this analysis. Table \ref{tab:4} presents the maximal absolute values of shearlet coefficients of an
$512\times 512$-image with a horizontal line ($t=0$) and different scales, both aligned with the line singularity ($c_{max}^0$), and
not aligned with the line singularity ($c_{max}^1$). The data in Table \ref{tab:4} does clearly indicate a significant difference of decay
rates between these two classes. We shall confirm this behavior from a different viewpoint, more precisely, a particular quantitative measure,
in Subsection~\ref{subsec:geometry}.
\begin{table}[ht]
\label{tab:4}
\caption{Maximal coefficients of a horizontal line}
\begin{tabular}{p{1.2cm}p{1.2cm}p{1.2cm}p{1.2cm}p{1.2cm}p{1.2cm}p{1.2cm}p{1.2cm}}
\toprule
$j$ & -1 & 0 & 1 & 2 & 3 & 4 &5\\
\midrule
$c_{max}^0$ & 0.160 & 0.098 & 0.068 & 0.038 &0.018 & 0.013 & 2.6E-3\\
\hline
$c_{max}^1$ & 0.149 & 0.048 & 0.007 & 1.9E-3 & 4.8E-4 & 2.3E-4 & 2.4E-5\\
\bottomrule
\end{tabular}
\end{table}

\subsection{Shear Invariance of the FDST}
\label{subsec:shearinvariance}

Since the shearing operator is quite distinctive in the definition of shearlets, one might ask whether
the FDST is in fact even shear invariant. In the continuum setting, when choosing $\psi$ to be a shearlet
generator as defined in \eqref{eq:psidef}, one can easily verify that, for any $f \in L^2(\mathbb{R}^2)$,
\[
\langle 2^{3j/2}\psi(S_k^{-1}A_{4^j}\cdot-m),f(S_s\cdot) \rangle = \langle 2^{3j/2}\psi(S_{k+2^{j}s}^{-1}A_{4^j}\cdot-m),f \rangle.
\]
This identity can be viewed as a manifestation of shear invariance, since it states that the shearlet coefficient of a sheared
image $f(S_s\cdot)$ at scale $j$, shear $k$, and spatial position $m$ equals the shearlet coefficient for the original image
$f$ at the same scale $j$, same spatial position $m$, but with a shear parameter shifted by $2^js$.

The following results shows that the FDST associated with digital shearlets has a similar property.

\begin{proposition}\label{prop:shearInvariance}
Let $I$ be an $N\times N$ image . Let $j$ be a scale, $s$ be a shear, $t$ with $|t|<1$ be a slope,
and $m=(m_1,m_2)$ be a spatial position such that $2^jt\in\bZ$ and $-2^j<s,s+2^jt<2^j$. Let $I_t:=I(S_t\cdot)$ be the sheared
image of $I$ such that ${\hat{I}_t}(\ox,\oy) =\hat{I}(\ox,\oy-t\ox)$ for all  $(\ox,\oy)\in\Sp\sigma_{j,s,m}^{\iota}$.
Then
\[
\langle \hat{I}_t,\sigma^{\iota}_{j,s,m}\rangle_{\Omega_R} = \langle \hat{I},\sigma^{\iota}_{j,s+2^jt,m}\rangle_{\Omega_R} \cdot e^{-2\pi i m_2\cdot2^jt}.
\]
\end{proposition}

\begin{proof}
It is sufficient to prove the claim for the case $\iota=21$, since the other cones can be handled similarly. In this case,
$
(\ox,\oy) = (\tfrac{2k}{R}, \tfrac{2k}{R} \cdot \tfrac{-2 \ell}{N})
$
and hence
\[
(\tilde{\omega}_x,\tilde{\omega}_y)  := (\ox,\oy-t\ox) = (\tfrac{2k}{R},  \tfrac{2k}{R} \cdot \tfrac{-2(\ell+tN/2)}{N}).
\]
Since $\hat{I}_t = \hat{I}((S_t^{-1})^T\cdot)$  on the support of $\sigma_{j,s,m}^{\iota}$, we have
\[
\begin{aligned}
&{\sqrt{|\cR_{j,s}|}}\cdot\langle \hat{I}_t, \sigma^{21}_{j,s,m}\rangle_{\Omega_R} \\
&= \hspace*{-0.25cm} \sum_{(\ox,\oy)\in\Omega_R^{21}}\hat{I}_t(\ox,\oy)\overline{W(4^{-j}\ox)V^j(s+2^j\tfrac{\oy}{\ox})}
e^{-2\pi i \ip{m}{(4^{-j} \ox,2^{j}\frac{\oy}{\ox})}}\\
& =\hspace*{-0.25cm}  \sum_{(\tilde{\omega}_x,\tilde{\omega}_y)\in (S_t^{-1})^T\Omega_R^{21}}\hspace*{-0.25cm} \hat{I}(\tilde{\omega}_x,\tilde{\omega}_y)
\overline{W(4^{-j}\tilde{\omega}_x)V^j(s+2^jt+2^j\tfrac{\tilde{\omega}_y}{\tilde{\omega}_x})}e^{-2\pi i \ip{m}{(4^{-j} \ox,2^{j}\frac{\oy}{\ox})}}\\
&= \hspace*{-0.25cm} \sum_{(\tilde{\omega}_x,\tilde{\omega}_y)\in (S_t^{-1})^T\Omega_R^{21}}\hspace*{-0.25cm} \hat{I}(\tilde{\omega}_x,\tilde{\omega}_y)
\overline{W(4^{-j}\tilde{\omega}_x)V^j(s+2^jt+2^j\tfrac{\tilde{\omega}_y}{\tilde{\omega}_x})}
e^{-2\pi i \ip{m}{(4^{-j}\tilde{\omega}_x,2^{j}\frac{\tilde{\omega}_y}{\tilde{\omega}_x}+2^jt)}}
\\&={\sqrt{|\cR_{j,s}|}}\cdot\langle \hat{I}, \sigma^{21}_{j,s+2^jt,m}\rangle e^{-2\pi i m_2 \cdot 2^jt}.
\end{aligned}
\]
This proves the claim.
\end{proof}


\section{Implementation of the FDST and its Inverse}
\label{sec:implementation}

After the formal introduction of the FDST and its theoretical analysis, we now turn to
discuss the details of our implementation, including the forward transform, the adjoint transform, and the
inverse transform. The associated code \url{ShearLab-PPFT-1.0} can be downloaded from \url{www.ShearLab.org}.

Figure~\ref{fig:flowcharts} provides an overview of the main steps of of the FDST and its inverse, which
were defined in Sections \ref{sec:STfinitedata} and \ref{sec:ISTfinitedata}, respectively.

\begin{figure}[ht]
\begin{center}
\includegraphics[height=2.1in]{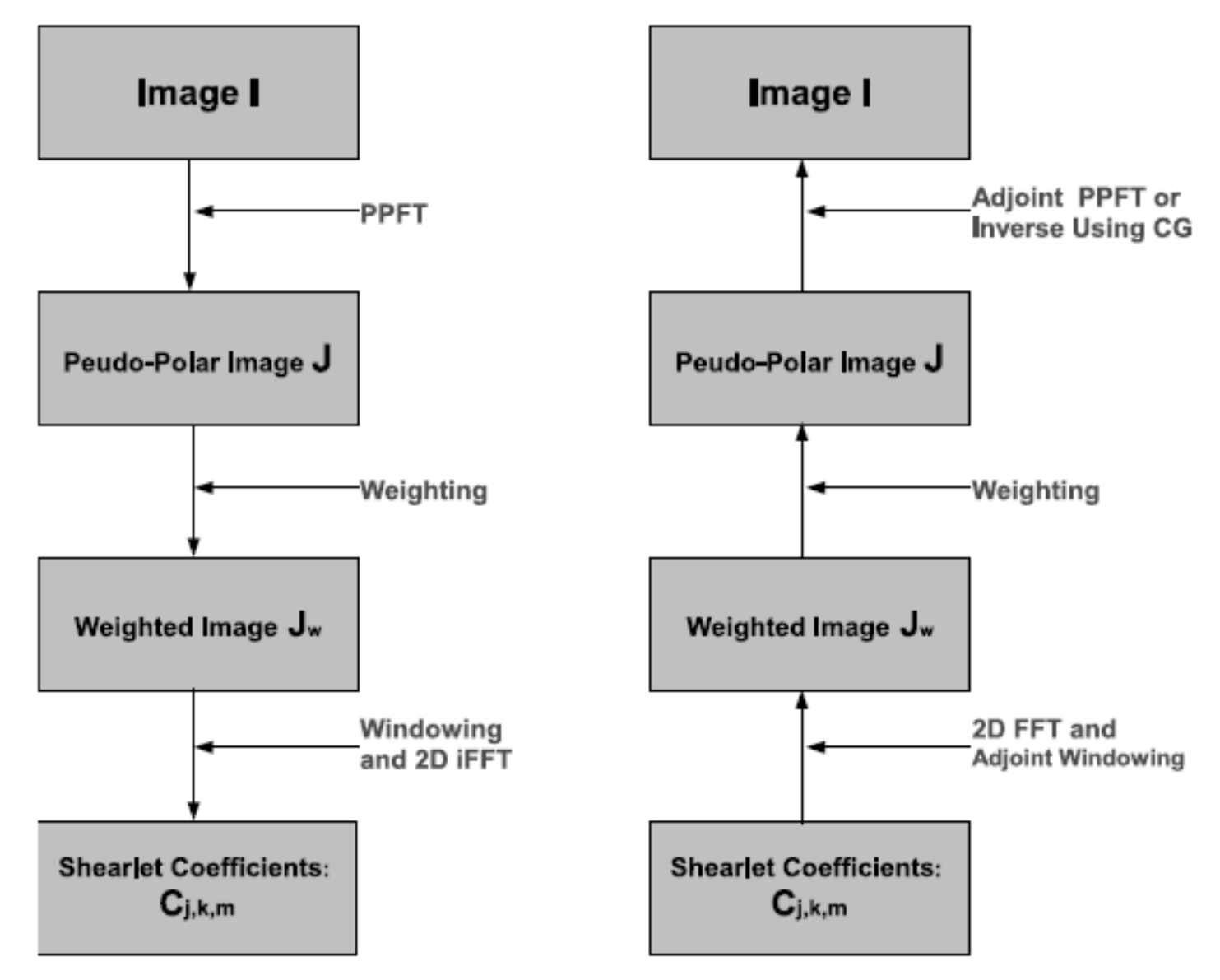}
\end{center}
\vspace{-0.5cm}
\caption{Flowcharts of the FDST (left) and its inverse (right).}
\label{fig:flowcharts}
\end{figure}

\subsection{Choices of Parameters}
\label{subsec:choiceW}

Before discussing the implementation details of each transform, let us start by elaborating on the choices of the
main parameters involved in our design of the FDST.

\subsubsection{Choice of the Parameter $m_0$}

The definition of the pseudo-polar Fou\-rier transform \eqref{def:ppft} has $m_0$ as a free parameter. In the paper
\cite{ACDIS08}, in which only the oversampling rate $R=2$ was considered, $m_0$ was chosen to be $2N+1$. This
particular choice enabled utilization of the 1D-FFT for the implementation of the fast pseudo-polar Fourier
transform. In the situation of an arbitrary oversampling rate $R$, which we consider, the fast pseudo-polar
Fourier transform is based on a (discrete) fractional Fourier transform.  In order to utilize the 1D-FFT along
one direction, $m_0$ has necessarily to be chosen as $\tfrac{2}{R}(RN+1)$; for details see Subsection \ref{subsec:forward}.
When choosing a different parameter $m_0$, the fractional Fourier transform needs to be applied on both directions, which
certainly would significantly lower the speed of the FDST. The (discrete) fractional Fourier transform can be implemented
with the same complexity as the 1D-FFT, but a different constant; in fact it is about 5 times slower than the 1D-FFT,
see \cite{Bailey:frFT}. To accelerate the speed, in our \url{ShearLab} package, the default $m_0$ is consequently set
to be $\tfrac{2}{R}(RN+1)$.

\subsubsection{Choice of Weight Function $w$}

A second free parameter is the weight function $w$, for which the criterion for isometry in
Theorem \ref{theo:weightedisometry} is required; and we discussed some choices in Subsection \ref{subsec:comparison}.
As could be seen, the performance in terms of almost isometry $M_{isom}$ differs depending on the
types of images which are considered. Hence the weight needs to chosen depending on the application.
It should be also emphasized that the particular type of weighting we considered in Subsection
\ref{subsec:comparison} provide good preconditioners for the conjugate gradient method, in case an
even more accurate inverse than the adjoint is required. In our \url{ShearLab} package, various
choices of $w$ are available.

\subsubsection{Choice of Window Functions $W_0, V_0$ and $W, V$}

The final main parameter is the choice of the window functions $W_0, V_0$ and $W, V$. In our implementation,
we use Meyer wavelets and define $W_0$ and $W$ to be the Fourier transform of the Meyer scaling function and
wavelet function, respectively, i.e.,
\[
W_0(\xi)=\left\{
\begin{array}{lcl}
1& : & |\xi|\leq\frac{1}{4},\\
\cos\left[\frac{\pi}{2}\nu(\frac{4}{3}|\xi|-\frac13)\right]&:& \frac{1}{4}\le|\xi|\le 1,\\
0 &:& \mbox{otherwise},
\end{array}\right.
\]
and
\[
W(\xi)=
\left\{\begin{array}{lcl}
\sin\left[\frac{\pi}{2}\nu(\frac43|\xi|-\frac13)\right]&:& \frac{1}{4}\le|\xi|\leq 1,\\
\cos\left[\frac{\pi}{2}\nu(\frac{1}{3}|\xi|-\frac13)\right]&:& 1\leq|\xi|\le 4,\\
0 &:& \mbox{otherwise},
\end{array}\right.
\]
where $\nu$ is a $C^k$ function or $C^\infty$ function such that
\begin{equation}\label{def:nu}
\nu(x)=\left\{\begin{array}{lcl}
0 &:& x\leq 0,\\
1-\nu(1-x) &:& 0\le x\le 1,\\
1 &:& x\geq 1.
\end{array}\right.
\end{equation}
In \url{ShearLab}, $\nu$ is set to be $\nu(x)=2x^2$ for $0\le x\le 1/2$ and $\nu(x)=1-2(1-x)^2$ for $1/2\le x\le 1$,
which is a $C^1$ function. Other choices are also available. The choice of $\nu$ then fixes $W_0$ and $W$.
Since $|W_0(\xi)|^2+|W(\xi)|^2=1$ for $|\xi|\le 1$, the required condition \eqref{eq:summabilityW}
is satisfied. These choices for $W_0$, $W$, and $\nu$ are illustrated in Figure~\ref{fig:Nu-W-W0}.
\begin{figure}[ht]
\begin{center}
\includegraphics[height=1in]{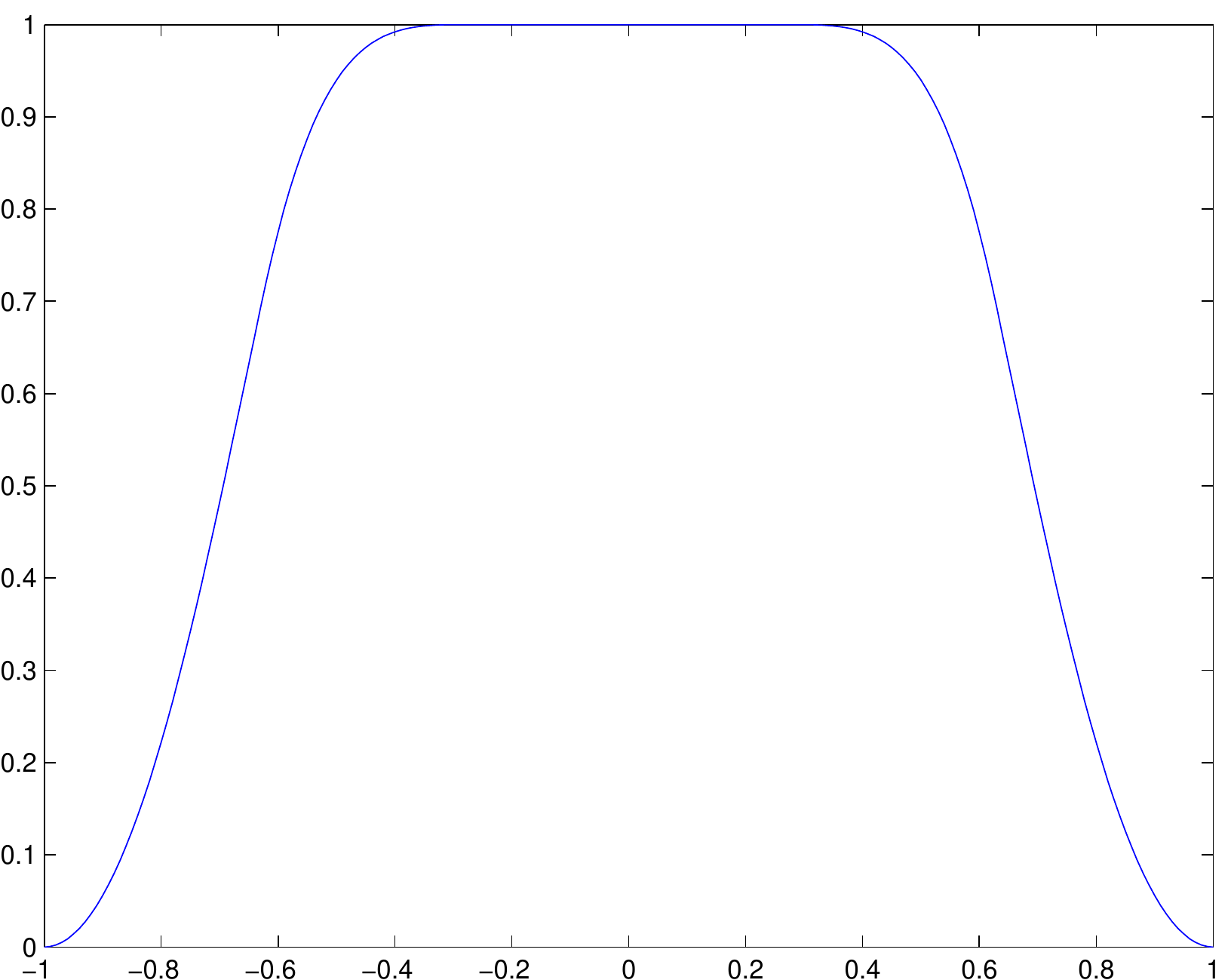}
\hspace*{1cm}
\includegraphics[height=1in]{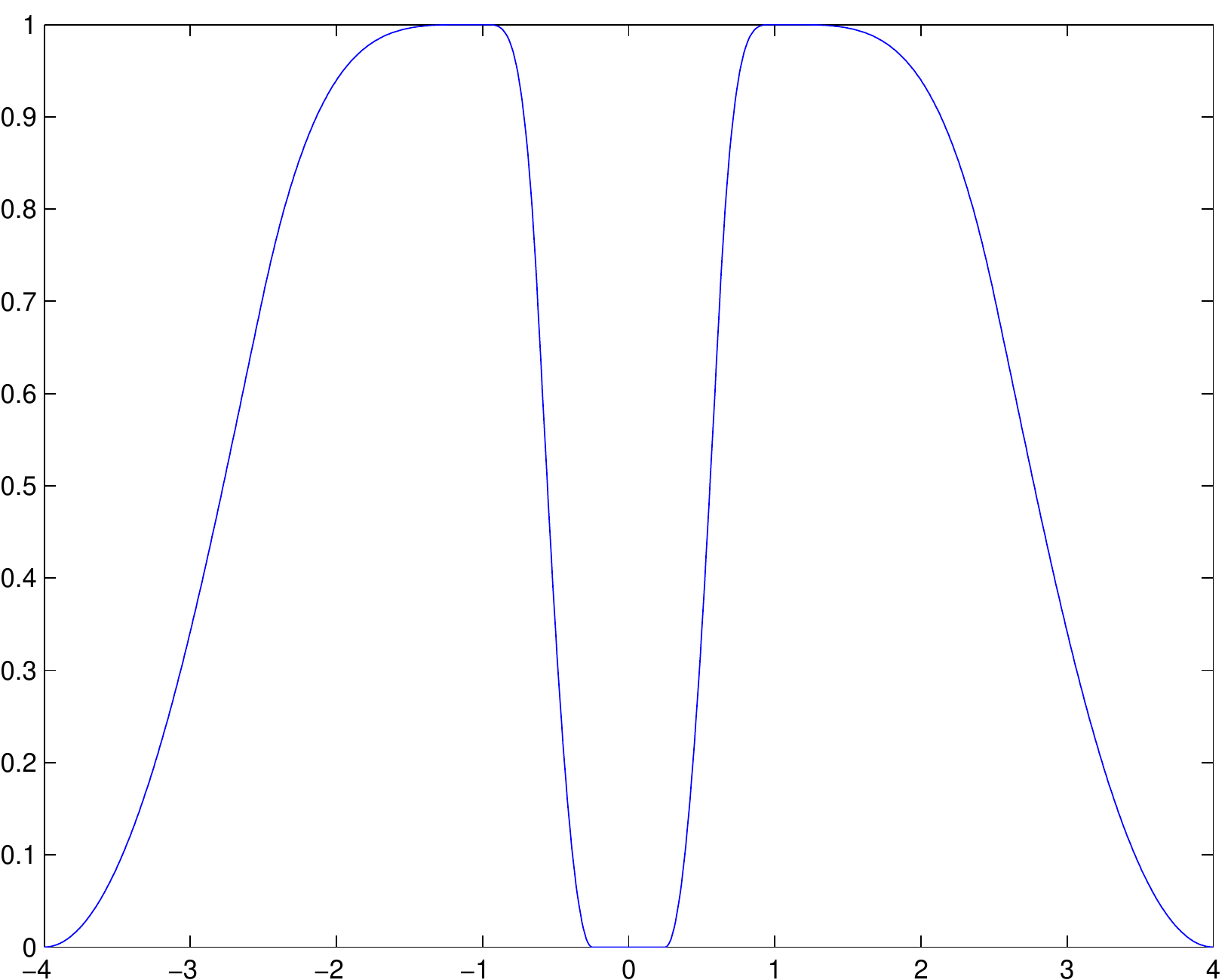}
\hspace*{1cm}
\includegraphics[height=1in]{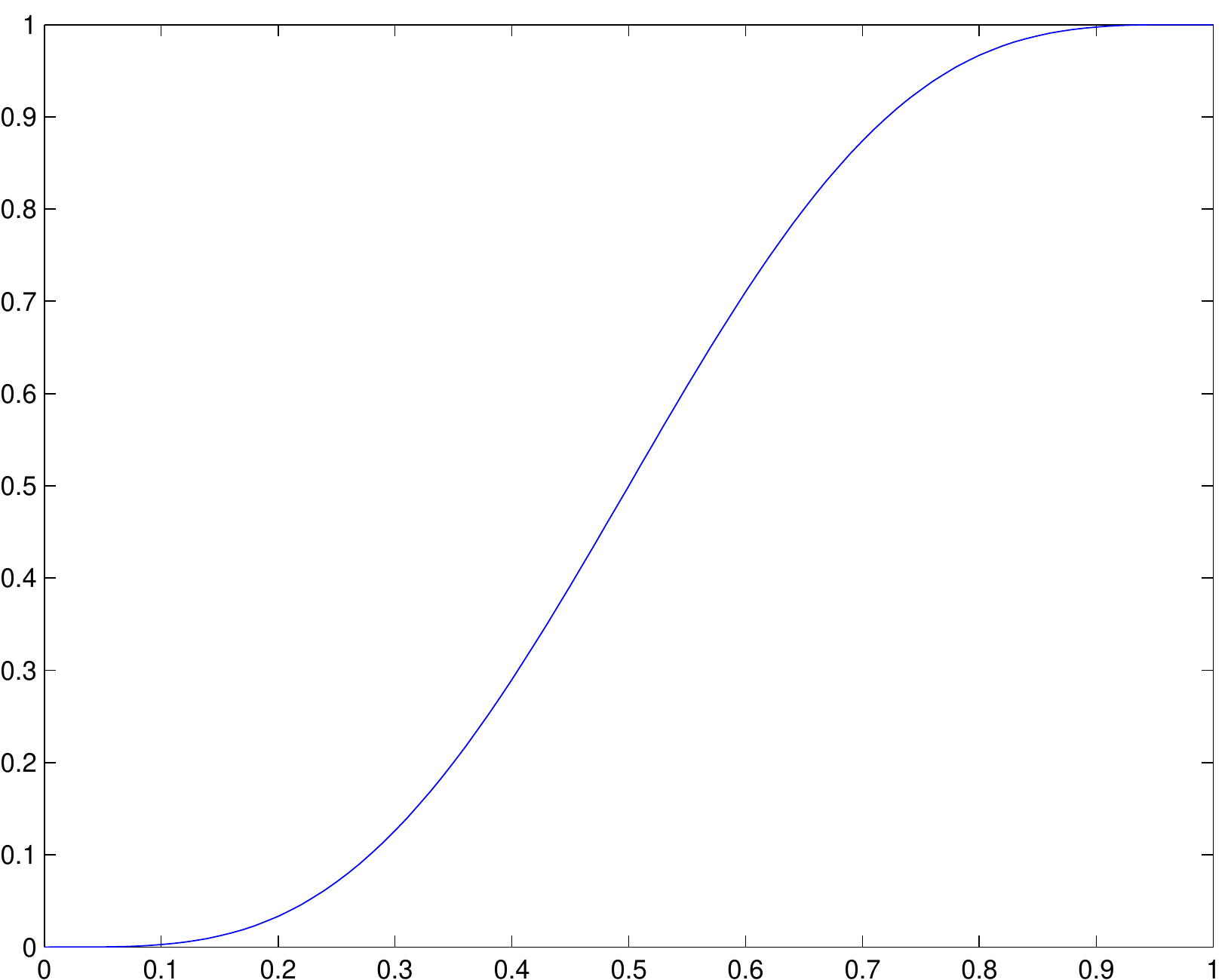}
\put(-306,-14){$W_0$}\put(-175,-14){$W$}\put(-44,-14){$\nu$}
\end{center}
\caption{The graphs of $W_0$, $W$, and $\nu$.}
\label{fig:Nu-W-W0}
\end{figure}

Provided $\nu$ satisfies \eqref{def:nu}, it can be used to design the bump function $V$ by setting
$V(\xi)=\sqrt{\nu(1+\xi)+\nu(1-\xi)}$ for $-1\le \xi\le 1$. This automatically satisfies \eqref{eq:summabilityV2},
and is our choice in the implementation. The function  $V_0$ is defined to be $V_0\equiv 1$.


\subsection{FDST}
\label{subsec:forward}

We now present our implementation of the FDST, which mainly consists of three parts:
The fast pseudo-polar Fourier transform, the weighting on the pseudo-polar grid, and the digital shearlet
windowing on the pseudo-polar grid followed by 2D-iFFT. These parts correspond to the operators $P$, $w$, and
$W$ introduced and discussed in Section~\ref{sec:ISTfinitedata}.
To facilitate our presentation of the implementation details, we require the following operators:

The \emph{(unaliased) fractional Fourier transform} (frFT) of a vector $c\in\bC^{N+1}$ with respect to a
fraction $\alpha \in \bC$ is defined to be
\[
(F_{N+1}^\alpha c)(k) :=\sum_{j=-N/2}^{N/2} c(j)e^{-2\pi i \cdot j\cdot k
\cdot\alpha}, \quad k = -\tfrac{N}{2}, \ldots, \tfrac{N}{2}.
\]
It was shown in \cite{Bailey:frFT}, that the fractional Fourier transform $F_{N+1}^{\alpha}c$ can be computed using
$O(N\log N)$ operations. In the special case of $\alpha = 1/(N+1)$, the fractional Fourier transform becomes the
(unaliased) 1D Fast Fourier Transform (1D-FFT), which in the sequel we will denote by $F_1$. Similarly, we will
denote the 2D Fast Fourier Transform (2D-FFT) by $F_2$, and the inverse of the $F_2$ by $F_2^{-1}$ (2D-iFFT).

Let now $N$ be even and $m>N$ be an odd integer. Then the {\em padding operator} $E_{m,n}$ acting on a vector $c\in\bC^{N}$
gives a symmetrically zero padding version of $c$ in the sense that
\[
(E_{m,N}c)(k) =
\begin{cases}
c(k) & k=-\tfrac{N}{2}, \ldots, \tfrac{N}{2}-1,\\
0     & k\in\{-\tfrac{m}{2}, \ldots, \tfrac{m}{2}\}\setminus\{-\tfrac{N}{2}, \ldots, \tfrac{N}{2}-1\}.
\end{cases}
\]

The fast pseudo-polar Fourier transform was introduced in \cite{ACDIS08}, however only for the oversampling rate $R=2$.
We will next extend this algorithm to an arbitrary oversampling rate and show that its complexity is also $O(N^2\log N)$.
For this, let $I$ be an image of size $N\times N$. We restrict to the cone $\Omega_R^1$. Choosing $m_0 = \frac{2}{R}(RN+1)$
to utilize the 1D-FFT, for $(\ox,\oy) \in  \Omega_R^1$,
\[
\begin{aligned}
\hat{I}(\ox,\oy)
&=\sum_{u,v=-N/2}^{N/2-1}I(u,v)e^{-\frac{2\pi i}{m_0}(u\ox+v\oy)}
=\sum_{u=-N/2}^{N/2-1}\sum_{v=-N/2}^{N/2-1}I(u,v)e^{-\frac{2\pi i}{m_0}(u\frac{-4k\ell}{RN}+v\frac{2k}{R})}\\
& = \sum_{u=-N/2}^{N/2-1}\left(\sum_{v=-N/2}^{N/2-1}I(u,v)e^{-\frac{2\pi i vk}{RN+1}}\right)e^{-{2\pi i u\ell}\cdot \frac{-2k}{(RN+1) \cdot N}}.
\end{aligned}
\]
The above identity shows that $\hat{I}$  on $\Omega_R^1$ can be obtained by performing 1D-FFT on the extension of $I$ along direction $v$ and then applying
the frFT along direction $u$. More precisely, let $\tilde I(u,\cdot) := E_{RN+1,N}I(u,\cdot)$, $-N/2\le u\le N/2-1$ be the symmetric zero padding of
$I$ and let $\tilde I_1(u,\cdot):=F_1\tilde I(u,\cdot)$ be the 1D FFT of $\tilde I$ along direction $v$.  Also, let
$\tilde I_2(\cdot,k) := E_{N+1,N}\tilde I_1(\cdot,k)$ for $k=-\tfrac{RN}{2}, \ldots, \tfrac{RN}{2}$ be the  symmetric zero padding of
$\tilde I_1$ along direction $u$. Note that, $\tilde I_1$ is then of size $(RN+1)\times N$ and $\tilde I_2$ is of size $(RN+1)\times (N+1)$.Then the above
identity can be written as
%
\[
\begin{aligned}
\hat{I}(\ox,\oy)& =\sum_{u=-N/2}^{N/2-1}\tilde I_1(u,k)e^{-{2\pi i u\ell}\cdot \frac{-k}{(RN+1) \cdot N/2}} =\sum_{u=-N/2}^{N/2}\tilde I_2(u,k)
e^{-{2\pi i u\ell}\cdot \frac{-2k}{(RN+1) \cdot N}}
\\& = (F_{N+1}^{\alpha_k}\tilde I_2(\cdot,k))(\ell),
\end{aligned}
\]
where  $\alpha_k=-\frac{k}{(RN+1)N/2}$ is the fraction in the frFT.

Thus, the pseudo-polar Fourier transform $\hat{I}(\ox,\oy), (\ox,\oy)\in\Omega_R^1$ -- similarly for $\Omega_R^2$ -- for
arbitrary oversampling rate $R$ can be computed by the steps described in Algorithm \ref{alg:ppft}.

\renewcommand{\algorithmiccomment}[1]{//#1}
\begin{algorithm}[H]
\caption{Fast Pseudo-Polar Fourier Transform of $I$ on $\Omega_R^1$}
\label{alg:ppft}
\begin{algorithmic}[1]\vspace*{0.2cm}
\item[(a)]{\bf Input}: Image $I$ of size $N\times N$.\\[1ex]
\item[(b)]{\bf Output}: The pseudo-polar Fourier transform $\hat{I}_{\Omega_R^1}$.\\[1ex]
\item[(c)] {\bf PPFT}: Image $I(u,v), -N/2\le u,v\le N/2-1$.
\STATE $\tilde I(u,\cdot)\leftarrow E_{RN+1,N}I(u,\cdot)$, $-N/2\le u\le N/2-1$. Symmetrically padding the image $I$ along direction $v$
to obtain an image $\tilde I$ of size $(RN+1)\times N$.

\STATE $\tilde I_1(u,\cdot)\leftarrow F_1\tilde I(u,\cdot)$, $-N/2\le u\le N/2-1$.
For each vector of $\tilde I$ along direction $v$, perform the 1D-FFT along direction $v$ to get $\tilde I_1$.

\STATE $\tilde I_2(\cdot,k)\leftarrow E_{N+1,N}\tilde I_1(\cdot,k)$, $-RN/2\le k\le RN/2$. Symmetrically padding the image $\tilde I_1$
along direction $u$ to get image $\tilde I_2$ whose size along direction $u$ is $N+1$.

\STATE $\hat I_{\Omega_R^1}\leftarrow F_{N+1}^{\alpha_k}\tilde I_2(\cdot,k)$, $-RN/2\le k\le RN/2$. Perform the
fractional Fourier transform with respect to a fraction $\alpha_k= -\frac{k}{(RN+1)N/2}$ along direction $u$.
\end{algorithmic}
\end{algorithm}

Since the 1D-FFT and 1D-frFT require only $O(N\log N)$ operations for a vector of size $N$, the total complexity of the pseudo-polar
Fourier transform is $O(N^2\log N)$ for an image of size $N\times N$. It should be emphasized that our Algorithm~1 is more simple and
efficient even in case $R=2$ than the algorithm described in \cite{ACDIS08}, in which $F_2$ is applied to $\tilde I$ followed by the
application of $F_1^{-1}$ to the resulting image in order to obtain $\tilde{I}_1$. The algorithm in \cite{ACDIS08} is slightly more
redundant while ours is more efficient by combining these two steps to one step. Moreover, the fractional Fourier transform along
direction $u$ is only of size $N+1$, as compared to size $RN+1$ in \cite{ACDIS08}.

We would like to also remark that, for a different choice of constant $m_0$, one can  compute the pseudo-polar Fourier transform also with complexity
$O(N^2\log N)$ for an image of size $N\times N$, in which case the fractional Fourier transform needs to be applied in both directions
$u$ and $v$ of the image. For example, for computing $\hat{I}$ on $\Omega_R^1$, the fractional Fourier transform is first applied to
$I$ along direction $v$ with a fixed fraction constant $\alpha = \tfrac{1}{m_0R/2}$, and then to the another direction $u$ with fractions
$\alpha_k =\tfrac{-k}{m_0R/2\cdot N/2}$ depending on $k$. In this case, it is, of course, slower than the special choice $m_0=\tfrac{2}{R}(RN+1)$,
since frFTs are involved in both directions while the special choice $m_0$ utilizes FFT.

The next step, which is an application of the weight function, is simply a point-wise multiplication on the pseudo-polar grid, i.e.,
\begin{itemize}
\item[(d)] {\bf Weighting}: Let $J:\Omega_R\rightarrow\bC$ be the pseudo-polar image of $I$ and $w:\Omega_R\rightarrow \bR^+$ be any suitable
weight function on $\Omega_R$. Compute the point-wise multiplication $J_w(\ox,\oy) = J(\ox,\oy)\cdot \sqrt{w(\ox,\oy)}$, $(\ox,\oy)\in\Omega_R$.
\end{itemize}

For the windowing, a sequence of subbands $\{\varphi^{\iota_0}_0: \iota_0\}\cup\{\sigma_{j,s,0}^{\iota}: j,s,\iota\}$ is computed, whose
frequency tiling covers the pseudo-polar grid. As discussed in Subsections~\ref{subsec:Windowing} and \ref{subsec:choiceW}, these subbands
are constructed using Meyer scaling functions and wavelets. The weighted image is then windowed by each subband, followed by an application
of the 2D-iFFT to each windowed subimage. This results in a sequence of coefficients -- the digital shearlet coefficients -- structured
block by block:

\begin{itemize}
\item[(e)]{\bf Subband Windowing}: Compute the digital shearlet coefficients by windowing $J_w$ with respect to the digital shearlets
$\{\varphi^{\iota_0}_n: \iota_0,n\}\cup\{\sigma_{j,s,m}^{\iota}: j,s,m,\iota\}$ defined in Definition~\ref{defi:digitalshearlets}.
\end{itemize}

For the detailed step-by-step description of the algorithm for the FDST used in \url{ShearLab}, we refer to
Algorithm~\ref{alg:forward} (see Appendix).


\subsection{Adjoint FDST}

The algorithm for the adjoint shearlet transform can be straightforwardly derived from the FDST, and the
main idea was already described in Section \ref{sec:ISTfinitedata}. It should be noted that the adjoint fractional Fourier
transform for a vector $c\in\bC^{N+1}$ with respect to a constant $\alpha\in\bC$ is  given by $F_{N+1}^{-\alpha}c$. Moreover,
for $m>N$ the adjoint operator $E_{m,N}^\star$ for the padding operator $E_{m,N}$ is given by $(E^\star_{m,N}c)(k) = c(k)$, $k=-N/2,\ldots,
N/2-1$ for a vector $c\in\bC^{m}$. The adjoint shearlet transform can be computed with a complexity $O(N^2\log N)$ similar
to the FDST, since it is obtained simply by `running the FDST backwards'.

For the detailed step-by-step description of the algorithm for the adjoint FDST used in \url{ShearLab}, we refer to
Algorithm~\ref{alg:adjoint}.

\subsection{ Inverse FDST}

The main idea to use the conjugate gradient method was already described in Section \ref{sec:ISTfinitedata}, and
for a detailed step-by-step description of the algorithm for the inverse FDST used in \url{ShearLab},
we refer to Algorithm~\ref{alg:inverse}.



\section{Quality Measures for Algorithmic Realization}
\label{sec:qualitymeasures}

To ensure and also prove that our implementation satisfies the previously proposed
desiderata, we will now define quality measures for each of those and in the sequel provide
numerical results on how accurate our implementation satisfies these. It is moreover
our hope that these measures shall serve as comparison measures for future implementations.
Although some measures are stated in `shearlet language', most are applicable to any
directional transform based on parabolic scaling.

In the following, $P$ shall denote the pseudo-polar Fourier transform defined by Algorithm \ref{alg:ppft}),
$w$ shall denote the weighting applied to the values on the pseudo-polar grid, $W$ shall be the windowing
with additional 2D-iFFT, $S$ shall denote the FDST defined by Algorithm \ref{alg:forward}, and $S^\star$
its adjoint defined by Algorithm \ref{alg:adjoint}.
We will further use the notation $G_{A}J$ for the solution of a matrix problem $AI=J$ using conjugate
gradient method with residual error set to be $10^{-6}$.

\vspace*{0.2cm}

\renewcommand{\labelenumi}{{\rm [D\arabic{enumi}]}}

\noindent[D1] {\sc Algebraic Exactness.}\\[1ex]
{\em \underline{Comments:}} We require the transform to be the precise implementation of a theory for digital
data on a pseudo-polar grid. In addition, to ensure numerical accuracy, we provide the following test,
which provides a quantitative measure for the closeness of the windows to form a tight frame.
\\[1ex]
{\em \underline{Measure:}} Generate a sequence of $5$ random images $I_1,\ldots,I_{5}$ (the integer $5$ is chosen for the purpose of fixing
the number of images to enable precise comparison) on the pseudo-polar grid for $N=512$ and $R=8$ with standard normally
distributed entries. Our quality measure will then be the Monte Carlo estimate for the operator norm
$\|W^\star W -  {\Id} \|_{op}$  given by
\[
M_{alg} = \max_{i=1,\ldots,5} \frac{\|W^\star W I_i - I_i \|_{2}}{\|I_i\|_2}.
\]

\noindent[D2] {\sc Isometry of Pseudo-Polar Fourier Transform.}\\[1ex]
{\em \underline{Comments:}} To ensure isometry of the utilized pseudo-polar Fourier transform, we introduced a careful
weighting of the pseudo-polar grid. The following test will now measure the closeness to being an isometry.
We expect to see a trade-off between the oversampling rate and the closeness to being an isometry. Since the
measure shall however serve as a common ground to compare different algorithms based on parabolic scaling,
we do not take the oversampling rate into account in the proposed measure. Instead, we would like to
remind the reader that this rate will instead affect the measure for speed [D6]. In the sequel, we will now
provide three different measures, each being designed to test a different aspect. \\[1ex]
{\em \underline{Measure:}}
\bitem
\item {\em Closeness to tightness.} Generate a sequence of $5$ random images $I_1,\ldots,I_{5}$
of size $512 \times 512$ with standard normally distributed entries. Our quality measure will then
be the Monte Carlo estimate for the operator norm $\|P^\star w P -  {\Id} \|_{op}$  given by
\[
M_{isom_1} = \max_{i=1,\ldots,5} \frac{\|P^\star w P I_i - I_i \|_{2}}{\|I_i\|_2}.
\]
\item {\em Quality of preconditioning.} Our quality measure will be
the spread of the eigenvalues of the Gram operator $P^\star w P$
given by
\[
M_{isom_2} = \frac{\lambda_{\max}(P^\star w P)}{\lambda_{\min}(P^\star w P)}.
\]
\item {\em Invertibility.}  Our quality measure will be the Monte Carlo estimate for the
invertibility of the operator $\sqrt{w} P$ using conjugate gradient method $G_{\sqrt{w}P}$ given by
\[
M_{isom_3} = \max_{i=1,\ldots,5} \frac{\|G_{\sqrt{w}P} \sqrt{w} P I_i - I_i \|_{2}}{\|I_i\|_2}.
\]
\eitem

\noindent[D3] {\sc Tight Frame Property.}\\[1ex]
{\em \underline{Comments:}} We now combine [D1] and [D2] to allow comparison with
other transforms, since those singleton tests might not be possible for any transform due to a
different inner structure.\\[1ex]
{\em \underline{Measure:}}
Generate a sequence of $5$ random images $I_1,\ldots,I_{5}$
of size $512 \times 512$ with standard normally distributed entries.
\bitem
\item {\em Adjoint transform.}
Our quality measure will be the Monte Carlo estimate for the operator norm $\|S^\star S -  {\Id} \|_{op}$
given by
\[
M_{tight_1} = \max_{i=1,\ldots,5} \frac{\|S^\star S I_i - I_i \|_{2}}{\|I_i\|_2}.
\]
\item {\em Inverse transform.}
Our quality measure will be the Monte Carlo estimate for the operator norm $\|G_{\sqrt{w}P} W^\star S -  {\Id} \|_{op}$
given by
\[
M_{tight_2} = \max_{i=1,\ldots,5} \frac{\|G_{\sqrt{w}P} W^\star S I_i - I_i \|_{2}}{\|I_i\|_2}.
\]
\eitem

\noindent[D4] {\sc Space-Frequency-Localization.}\\[1ex]
{\em \underline{Comments:}} The purpose of this test is to provide quantitative measures for the
degree to which the windows -- in our case the digital shearlets --  are localized in both space
and frequency. For this, we numerically measure mathematically precise notions of both decay and
smoothness.\\[1ex]
{\em \underline{Measure:}} Let $I$ be a shearlet in a $512 \times 512$ image centered at the origin
$(257,257)$ with slope $s=0$ at scale $j=3$, which in our case is the shearlet $\sigma_{3,0,0}^{11}+\sigma_{3,0,0}^{12}$.
Our quality measure will then be four-fold:
\bitem
\item {\em Decay in spatial domain.} Compute the decay rates $d_1,\ldots,d_{512}$
along lines parallel to the $y$-axis starting from the line $[257,\;:\;]$ and the decay rates
$d_{512}$, $\ldots$, $d_{1024}$ with $x$ and $y$ interchanged in the following way: Consider
exemplarily the line $[257:512,1]$. First compute the smallest monotone majorant $M(x,1)$, $x=257,\ldots,512$
 -- note that we could have also chosen a different `envelope' --
for the curve $|I(x,1)|$, $x=257,\ldots,512$. Then the decay rate along the line $[257:512,1]$ is defined
to be the average slope of the line, which is a least square fit to the curve $\log(M(x,1))$, $x=257,\ldots,512$.
Based on these decay rates, we choose our measure to be the average of the decay rates given by
\[
M_{decay_1} = \frac{1}{1024}\sum_{i=1,\ldots,1024} d_i.
\]
\item {\em Decay in frequency domain.} To check whether the Fourier transform of $I$
is compactly supported and to check its decay rate, let $\hat{I}$ be the 2D-FFT of $I$ and compute the
decay rates $d_i$, $i=1,\ldots,1024$ as before. Then we define the following two measures:
\bitem
\item[$\Diamond$] {\em Compactly supportedness.}
\[
M_{supp} = \frac{\max_{|u|,|v|\le 3}|\hat I(u,v)|}{\max_{u,v}|\hat I(u,v)|}.
\]
\item[$\Diamond$] {\em Decay rate.}
\[
M_{decay_2} = \frac{1}{1024}\sum_{i=1,\ldots,512} d_i.
\]
\eitem
\item {\em Smoothness in spatial domain.} Smoothness will be measured by the average of local H\"older regularity.
For this, for each $(u_0,v_0) \in \{1, \ldots, 512\}^2$, compute $M(u,v)= |I(u,v)-I(u_0,v_0)|$, $0<\max\{|u-u_0|,|v-v_0|\}\le 4$. Then the local
H\"older regularity $\alpha_{u_0,v_0}$ is the least square fit to the curve $(u,v) \mapsto \log(|M(u,v)|)$.
The measure is given by
\[
M_{smooth_1} = \frac{1}{512^2}\sum_{u,v}\alpha_{u,v}.
\]
\item {\em Smoothness in frequency domain.} Smoothness will again be measured by the average of local H\"older regularity.
Proceed as for measuring the smoothness in spatial domain now applied to $\hat{I}$, the 2D-FFT of $I$, to derive the
local H\"older regularity $\alpha_{u,v}$ for each $(u,v) \in \{1, \ldots, 512\}^2$. The measure is then given by
\[
M_{smooth_2} = \frac{1}{512^2}\sum_{u,v}\alpha_{u,v}.
\]
\eitem

\noindent[D5] {\sc True Shear Invariance.}\\[1ex]
{\em \underline{Comments:}} This test shall provide a measure for how close the transform is to being shear invariant.
In our case, the theory gives
\[
\ip{2^{3j/2} \psi(S_k^{-1} A_{4^{j}} \cdot -m)}{f(S_s \cdot)}
= \ip{2^{3j/2} \psi(S_{k+2^{j} s}^{-1} A_{4^{j}} \cdot -m)}{f},
\]
and we expect to see this behavior in the shearlet coefficient as discussed in Subsection~\ref{subsec:shearinvariance}.\\[1ex]
{\em \underline{Measure:}} Let $I$ be an $256 \times 256$ image with an edge through the origin $(129,129)$ of slope $0$.
Fix $s=1/2$, generate an image $I_s:=I(S_s\cdot)$, and let $j\in\{1,2,3,4\}$ be a scale. Note that $2^js\in\bZ$ for each of $j\in\{1,2,3,4\}$.
Our quality measure will then be the curve
\[
M_{shear,j} = \max_{-2^j<k,k+2^js<2^j} \frac{\|C_{j,k}(S I_s)  - C_{j,k+2^js}( S I)\|_2}{\|I\|_2}, \qquad \text{scale } j=1,2,3,4,
\]
where $C_{j,k}$ is the set of coefficients -- in our case the shearlet coefficient --  at scale $j$ and at shear index $k$.

\vspace*{0.2cm}

\noindent[D6] {\sc Speed.}\\[1ex]
{\em \underline{Comments:}}
When testing the speed, not only the asymptotic behavior will be measured but also the involved constants.
It should further be mentioned that for practical purposes it is usually sufficient to analyze the speed
up to a size of $N=512$.\\[1ex]
{\em \underline{Measure:}} Generate a sequence of $5$ random images $I_i$, $i=5,\ldots,9$ of size $2^i \times 2^i$
with standard normally distributed entries. Let $s_i$ be the speed of the transform $S$ applied to
$I_i$.
Our hypothesis is that the speed behaves like $s_i = c \cdot (2^{2i})^d$; $2^{2i}$ being the
size of the input. Let now $\tilde{d}_a$ be the average slope of the line,
which is a least square fit to the curve $i \mapsto \log(s_i)$.
Let also $f_i$ be the 2D-FFT applied to $I_i$, $i=5,\ldots,9$. Our quality measure will then be three-fold:
\bitem
\item {\em Complexity.}
\[
M_{speed_1} = \frac{\tilde{d}_a}{2\log 2}.
\]
\item {\em The constant.}
\[
M_{speed_2} = \frac15 \sum_{i=5}^{9} \frac{s_i}{(2^{2i})^{M_{speed_1}}}.
\]
\item {\em Comparison with 2D-FFT.}
\[
M_{speed_3} = \frac15 \sum_{i=5}^{9} \frac{s_i}{f_i}.
\]
\eitem

\noindent[D7] {\sc Geometric Exactness.}\\[1ex]
{\em \underline{Comments:}} Geometric objects such as edges should be as precise as possible be resembled by the
coefficients in the sense that analyzing the decay of the coefficient should detect such features. For the
FDST, these properties were theoretically analyzed in Subsection~\ref{subsec:line}.
Our model will be an image containing one line of a particular slope.\\[1ex]
{\em \underline{Measure:}} Let $I_1,\ldots, I_8$ be $256 \times 256$ images of an edge through the origin $(129,129)$ and of slope $[-1,-0.5,0,0.5,1]$
and the transpose of the middle three, and let $c_{i,j}$ be the associated shearlet coefficients
for image $I_i$ at scale $j$. Our quality measure will be two-fold:
\bitem
\item {\em Decay of significant coefficients.}
Consider the curve
\[
\frac18 \sum_{i=1}^8 \max{|c_{i,j} \text{(of shearlets aligned with the line)}|}, \qquad \text{scale } j,
\]
let $d$ be the average slope of the line, which is a least square fit to the logarithm of this curve, and define
\[
M_{geo_1} = d.
\]
\item {\em Decay of insignificant coefficients.}
Consider the curve
\[
\frac18 \sum_{i=1}^8 \max{|c_{i,j} \text{(of all other shearlets)}|}, \qquad \text{scale } j,
\]
let $d$ be the average slope of the line, which is a least square fit to the logarithm of this curve, and define
\[
M_{geo_2} = d.
\]
\eitem

\noindent[D8] {\sc Robustness.}\\[1ex]
{\em \underline{Comments:}} Two different types of robustness will be analyzed which
we believe are the most common impacts on a sequence of transform coefficients. We wish
to mention that we certainly also could have considered additional manipulations of the
coefficients such as deletions; but to provide sufficiently many tests balanced with a
reasonable testing time, we decided to restrict to those two.
\\[1ex]
{\em \underline{Measure:}}
\bitem
\item {\em Thresholding.} Let $I$ be the regular sampling of a Gaussian function with mean 0 and variance $256$
on $\{-128,...,127\}^2$ generating an $256 \times 256$-image. The quality measure for $k=1,2$ will be the curve
\[
M_{thres_{k,p_k}} = \frac{\|G_{\sqrt{w}P} W^\star \; {\rm thres}_{k,p_k} \, S I  - I\|_2}{\|I\|_2},
\]
where
\bitem
\item ${\rm thres}_{1,p_1}$ discards $100 \cdot (1-2^{-p_1})$ percent of the coefficients with $p_1 = [2:2:10]$),
\item ${\rm thres}_{2,p_2}$ sets all those coefficients to zero with absolute values below
the threshold $m/2^{p_2}$ with $m$ being the maximal absolute value of all coefficients with $p_2 = [0.001:0.01:0.05]$).
\eitem
\item {\em Quantization.} Let $I$ be the regular sampling of a Gaussian function with mean 0 and variance $256$
on $\{-128,...,127\}^2$ generating an $256 \times 256$-image.  The quality measure
will be the curve
\[
M_{quant,q} = \frac{\|G_{\sqrt{w}P} W^\star \; {\rm quant}_{q} \, S I  - I\|_2}{\|I\|_2}, \qquad q =[8:-0.5:6],
\]
where
${\rm quant}_q(c)={\rm round}(c/(m/2^q)) \cdot (m/2^q)$ and
$m$ being the maximal absolute value of all coefficients.
\eitem


\section{Numerical Evidence}
\label{sec:numerics}


In this section, we provide numerical results for the tests [D1]--[D8] detailed in
Section \ref{sec:qualitymeasures} of our
present implementation, which we consider as having reached a mature state after careful tuning the
parameters depending on these performance measures.
The associated code \url{ShearLab-PPFT-1.0} can be downloaded from \url{www.ShearLab.org}.

\subsection{Results for Tests [D1]--[D3]}

Table \ref{tab:5} presents the performance with respect to the quantitative measures in [D1]--[D3].

\begin{table}[ht]
\label{tab:5}
\caption{Results for [D1]--[D3]}
\begin{tabular}{p{1.7cm}p{1.7cm}p{1.7cm}p{1.7cm}p{1.7cm}p{1.9cm}}
\toprule
$M_{alg}$  &  $M_{isom_1}$  &  $M_{isom_2}$  & $M_{isom_3}$  &  $M_{tight_1}$  &  $M_{tight_2}$ \\
\midrule
6.6E-16  &  9.3E-4  &  1.834  & 3.3E-7  &  9.9E-4  &  3.8E-7\\
\bottomrule
\end{tabular}
\end{table}

The quantity $M_{alg}\approx$ 6.6E-16 confirms that the $\cD\cS\cH$ defined in Definition~\ref{defi:digitalshearlets}
is indeed up to machine precision a tight frame.

The slight tightness deficiency of $M_{tight_1}\approx$ 9.9E-4 (also $M_{isom_1}\approx$ 9.3E-4) mainly results from the isometry deficiency
of the weighting. However, for practical purposes this transform can be still considered to be an isometry allowing the utilization of the
adjoint as inverse transform. Progress on the choice of weights will further improve this measure. Observe though that there is a trade-off
between the sophistication of the weights, the running time of $S$, and the smoothness of the shearlets.

Further, note that the condition number ($M_{isom_2} \approx 1.834$) of the Gram matrix is quite close to 1, which -- in case an even higher
accurate inverse than the adjoint is required -- allows us to employ the conjugate gradient method very efficiently for computing the inverse
of the FDST ($M_{isom_3}\approx $ 3.3E-7 and $M_{tight_2}\approx$ 3.8E-7).

\subsection{Space-Frequency-Localization Test [D4]}

The reference shearlet $I$ required for [D4] is illustrated in Figure~\ref{fig:Shearlet} in both the spatial and frequency domain.

\begin{figure}[ht]
\begin{center}
\includegraphics[height=1.0in]{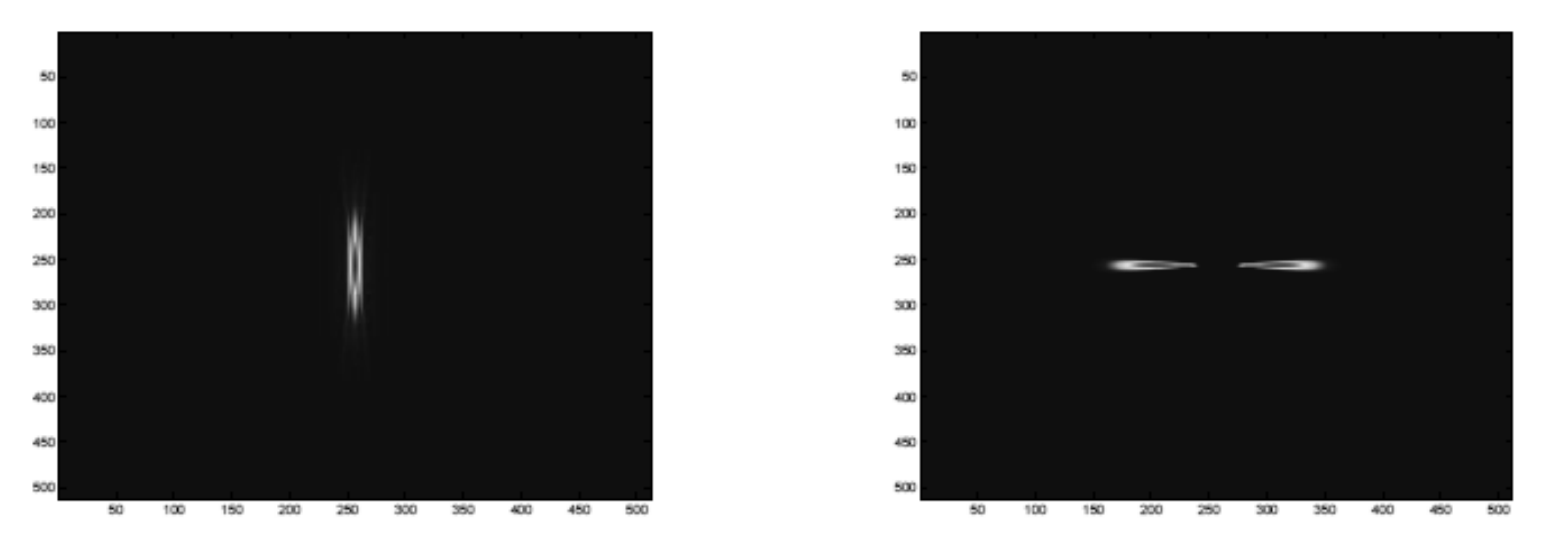}
\end{center}
\caption{Graphs of the reference shearlet for test [D4] centered at origin at scale 3 and shear parameter 0 in both spatial and frequency
domain with $N=512$.}
\label{fig:Shearlet}
\end{figure}

Table \ref{tab:6} presents the space-frequency-localization measures $M_{decay_1}$, $M_{decay_2}$, and $M_{supp}$ for analyzing the decay in
spatial and frequency domains as well as the measures $M_{smooth_1}$ and $M_{smooth_2}$ for smoothness in spatial and frequency domains.

\begin{table}[ht]
\label{tab:6}
\caption{Results for [D4]}
\begin{tabular}{p{2.2cm}p{2.2cm}p{2.2cm}p{2.2cm}p{2.2cm}}
\toprule
$M_{decay_1}$ & $M_{supp}$ & $M_{decay_2}$& $M_{smooth_1}$ & $M_{smooth_2}$ \\
\midrule
-1.920 & 5.5E-5 & -3.257 & 1.319 & 0.734\\
\bottomrule
\end{tabular}
\end{table}

The measurements indicate that the shearlet illustrated in Figure~\ref{fig:Shearlet} decays slower in spatial domain ($M_{decay_1} \approx$ -1.920)
than in the frequency domain ($M_{decay_2}\approx$ -3.257) in terms of the average decay rates. In terms of average local H\"older smoothness,
the shearlet in spatial domain is smoother than in the frequency domain proven by the fact that $M_{smooth_1} \approx 1.319>0.734 \approx M_{smooth_2}$.
The very small value $M_{supp}\approx$ 5.5E-5 -- although it is not zero due to round off errors -- indicates that our shearlet is indeed compactly
supported.

\subsection{Shear Invariance Test [D5]}

Test [D5] requires a reference image $I$ and a sheared version $I_s = I(S_s\cdot)$, which is illustrated in Figure \ref{fig:edges}
for $s=0.5$ alongside with their pseudo-polar Fourier transforms $\hat{I}$ and $\hat{I}_s$, respectively.

\begin{figure}[ht]
\begin{center}
\includegraphics[height=1.0in]{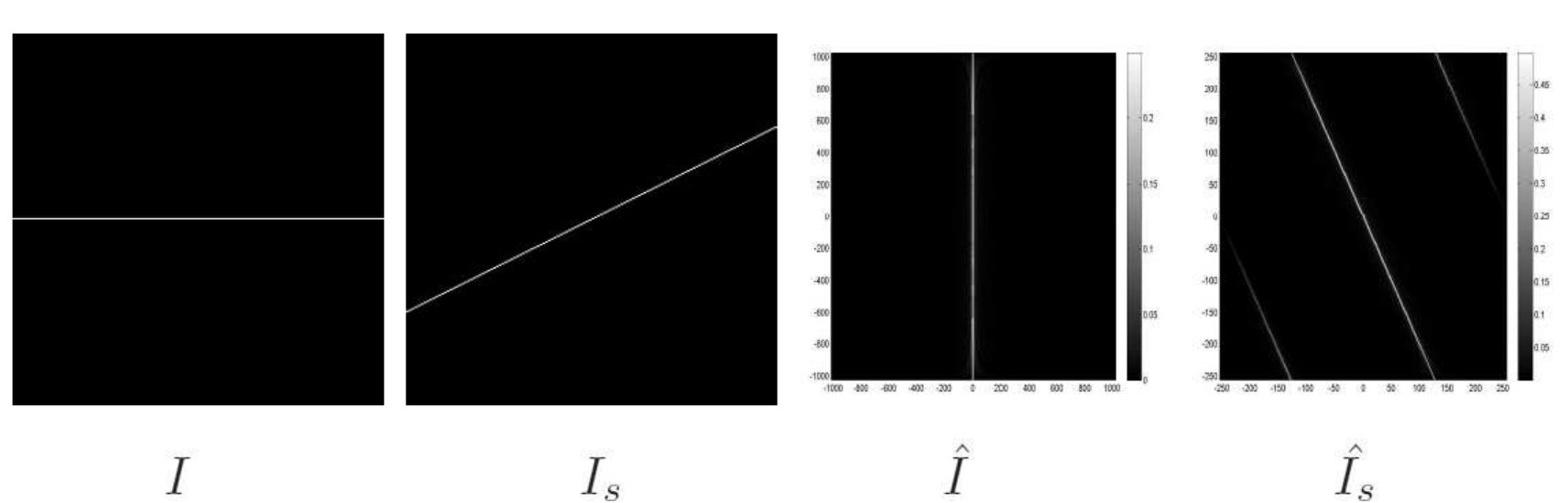}
\end{center}
\caption{Graphs of the reference image $I$ and its sheared version $I_s$ for $s = 0.5$ in spatial domain and $\hat{I}
$ and $\hat{I_s}$ in pseudo-polar domain.}
\label{fig:edges}
\end{figure}

The shear invariance measurements are presented in Table \ref{tab:7}.
\begin{table}[ht]
\label{tab:7}
\caption{Results for [D5]}
\begin{tabular}{p{2.2cm}p{2.2cm}p{2.2cm}p{2.2cm}p{2.2cm}}
\toprule
Scale & 1 & 2 & 3 &4 \\ \midrule
$M_{shear,j}$ & 1.6E-5 & 1.8E-4 & 0.002  & 0.003\\
\bottomrule
\end{tabular}
\end{table}
This table shows that the FDST is indeed almost shear invariant. A closer inspection shows that
$M_{shear,1}$ and $M_{shear,2}$ are relatively small compared to the measurements with respect to finer scales $M_{shear,3}$ and
$M_{shear,4}$. The reason for this is the aliasing effect which shifts some energy to the high frequency part near the boundary
away from the edge in the frequency domain, see also the graph of $\hat{I}_s$ in Figure \ref{fig:edges}.

\subsection{Speed Test [D6]}

The measurements for speed performance, i.e., for the complexity $M_{speed_1}$, the constant $M_{speed_2}$, and the comparison with
2D-FFT $M_{speed_3}$, are presented in Table \ref{tab:8}.

\begin{table}[ht]
\label{tab:8}
\caption{Results for [D6]-[D7]}
\begin{tabular}{p{2.2cm}p{2.2cm}p{2.2cm}p{2.2cm}p{2.2cm}}
\toprule
$M_{speed_1}$ & $M_{speed_2}$ & $M_{speed_3}$ & $M_{geo_1}$ & $M_{geo_2}$\\
\midrule
1.156 & 9.3E-6 & 280.560 &-1.358 & -2.032\\
\bottomrule
\end{tabular}
\end{table}

%
%

The complexity constant $M_{speed_1}\approx 1.156$
confirms that the order of complexity of our FDST algorithm is indeed close to linear (compare $O(N^2\log N)$ in theory). The constant of comparison with
2D-FFT is $M_{speed_3}\approx 280$, which seems significantly slower than the usually 2D-FFT. However, we should notice that the 2D-FFT is applied directly
to the image of size $N\times N$ while the FDST employs fractional Fourier transforms and subband windowing on a oversampling pseudo-polar grid of size
$2\times (RN+1)\times (N+1)$ with oversampling rate $R=8$. Taking into account that the fractional Fourier transform is about 5 times slower than the FFT
and that the redundancy (about $4R$) comes from subband windowing with oversampling rate $R=8$, we conclude that the constant $M_{speed_3}$ is reasonable
and our implementation is indeed comparable with the 2D-FFT.

\subsection{Geometric Exactness Test [D7]}
\label{subsec:geometry}

Test [D7] requires reference images with edges at various slopes which are plotted in Figure \ref{fig:GeoExat}.
\begin{figure}[ht]
\begin{center}
\includegraphics[height=1.1in]{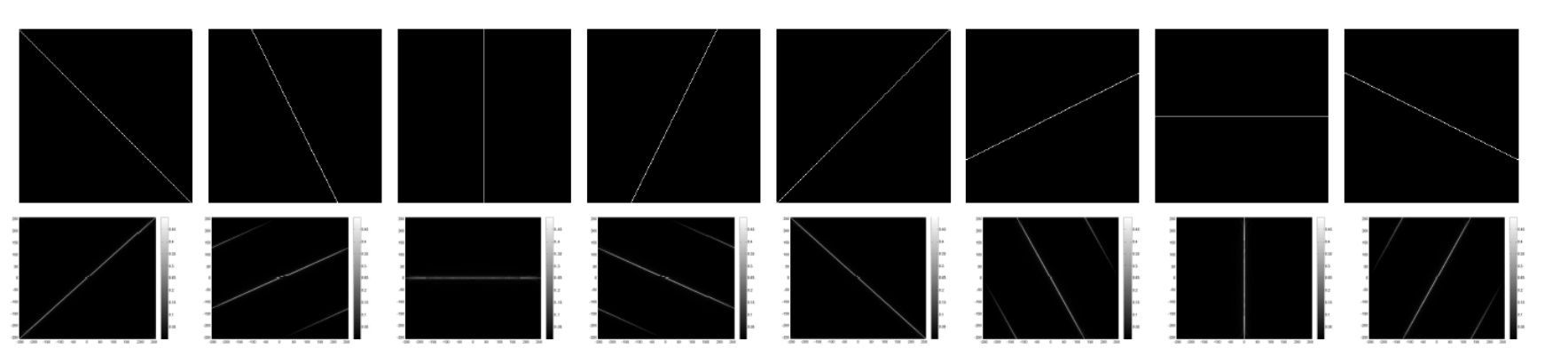}
\end{center}
\caption{The graphs of edges in both spatial and frequency domains. Top 8: Edges in spatial domain. Bottom 8: Edges in frequency domain.}
\label{fig:GeoExat}
\end{figure}

The associated geometric exactness measurements $M_{geo_1}$ and $M_{geo_2}$ are presented in Table \ref{tab:8}.
These two measurements  show the differences between the
decay rates of the shearlet coefficients aligned with edges (significant shearlet coefficients) and the shearlet coefficients
not aligned with the edges (insignificant shearlet coefficients).
As theoretically proven in Subsection \ref{subsec:line}, and hence expected for the implementation,
the insignificant shearlet coefficients ($M_{geo_2}\approx$ -2.032 ) decay much faster than the significant shearlet coefficients
with increasing scale $j$.

\subsection{Robustness Test [D8]}

Table \ref{tab:10} presents the measurements for the robustness test [D8].

\begin{table}[ht]
\label{tab:10}
\caption{Results for [D8]}
\begin{tabular}{p{1.7cm}p{1.7cm}p{1.7cm}p{1.7cm}p{1.7cm}p{1.9cm}}
\toprule
$M_{thres_{1,p_1}}$&1.5E-8&7.2E-8&2.5E-5 & 0.001  & 0.007 \\
\midrule
$M_{thres_{2,p_2}}$& 0.005 & 0.039 & 0.078 & 0.113  & 0.154 \\
\midrule
$M_{quant,q}$& 0.034 & 0.047 & 0.057 & 0.071 & 0.109\\
\bottomrule
\end{tabular}
\end{table}

These results show that even if $100(1-2^{-10})\approx 99.9\%$ of the shearlet coefficients are discarded, the original
image is still well approximated by the reconstructed image ($M_{thres_{1,p_1}}\approx$ 0.007).  Thus the number of
the significant coefficients is relatively small compared to the total number of shearlet coefficients. The second row
indicates that knowledge of the shearlet coefficients with absolute value greater than $m(1-1/2^{0.001})$ -- hence about
$0.1\%$ of coefficients -- is sufficient for precise reconstruction ($M_{thres_{1,p_2}}\approx$ 0.005).
The quantization test $M_{quant,q}$ attests the FDST high resilience against even quite coarse quantization.


\section{Conclusions}

We described a natural digitization of the continuous shearlet transform based on the pseudo-polar Fourier transform. For this,
we first introduced a digital shearlet tight frame, which is an exact digitization of the classical band-limited system defined in
the continuum domain. This shows that the
shearlet framework provides a unified treatment of the continuum and digital realm. We then described a fast digital
shearlet transform (FDST) whose main ingredients are a weighted pseudo-polar Fourier transform to achieve an isometric mapping
into the pseudo-polar domain and a windowing based on the digital shearlet tight frame. We discussed several
choices of weighting functions to achieve almost isometry so that inverse digital shearlet transform can be
obtained by using the adjoint transform. Various properties of the FDST are proven,
and details of its implementation are provided. We further define several measures that quantify the performance
of the FDST, thereby justifying the term `rational design'. These measures also provide a common ground to compare
the performance of different algorithms based on parabolic scaling.
The FDST as well as the performance measures are publicly available in the software
package \url{ShearLab}, see \url{www.ShearLab.org}.



\appendix
\section{Algorithms}
\renewcommand{\thesection}{A}

In this appendix, we provide pseudo-codes in Matlab style for algorithms of FDST (Algorithm \ref{alg:forward}),
of the adjoint FDST (Algorithm \ref{alg:adjoint}), and the inverse FDST (Algorithm \ref{alg:inverse}).
For the required notation, we refer to Section \ref{sec:implementation}.




\renewcommand{\thealgorithm}{\thesection.1}

\renewcommand{\algorithmiccomment}[1]{//#1}
{\allowdisplaybreaks
\begin{algorithm}[h]
\caption{Fast Digital Shearlet Transform (FDST)}
\label{alg:forward}
{\allowdisplaybreaks
\begin{algorithmic}[1]\vspace*{0.2cm}
\item[{\rm(a)}] {\bf Input:} Image $I=\{[I]_{u,v}: -N/2\le u,v\le N/2-1\}$ of size $N\times N$, the oversampling rate $R$,
and the precomputed weight matrix $w$  of size $2\times (RN+1)\times (N+1)$.\\[1ex]

\item[{\rm(b)}]{\bf Output:} Digital shearlet coefficients $C_I =\{c^{\iota_0}:\iota_0\}\cup \{c_{j,s}^{\iota} : j, s,  \iota\}$.
Each $c_{j,s}^\iota$ is a matrix of size $\cL_{j}^1\times \cL_{j,s}^2$, see Definition~\ref{defi:digitalshearlets}.\\[1ex]

\item[{\rm(c)}] {\bf PPFT:} The pseudo-polar Fourier transform (see also Algorithm~\ref{alg:ppft}).
\STATE $J\leftarrow zeros(2,RN+1,N+1)$. $m\leftarrow RN+1$. \COMMENT{\texttt{$J$ is the grid $J= J_{\Omega_R^1}\cup J_{\Omega_R^2}$}.}

\FOR{$sector = 1$ to $2$}

\STATE {\bf if} {$sector = 2$} {\bf then}
$I = I^T$ 
{\bf end if}

\FOR{$v=-N/2$ to $N/2-1$}
\STATE $q\leftarrow [I]_{:,v}$.
, $q\leftarrow E_{RN+1,N}q$
, $[J]_{sector, :, v}\leftarrow F_1q$.

 \COMMENT{\texttt{symmetric zero-padding and FFT along direction $u$.}}
\ENDFOR

\FOR{$k=-RN/2$ to $RN/2$}

\STATE  $q\leftarrow[J]_{sector,k,:}$, $\alpha \leftarrow -\frac{k}{mN/2}$, $q\leftarrow F_{N+1}^\alpha q$, $[J]_{sector,k,:}\leftarrow q$.

\COMMENT{\texttt{frFT along direction $v$.}}
\ENDFOR

\ENDFOR\\[1ex]

\item[{\rm(d)}]{\bf Weighting:} Weighting on the pseudo-polar grid with precomputed weight $w$.
\STATE $J \leftarrow J{\rm.*}\sqrt{w}$. \COMMENT{\texttt{.* is the point-wise multiplication}.}\\[1ex]

\item[{\rm(e)}]{\bf  Subband Windowing:} Subband windowing  with $\cD\cS\cH$.

\STATE $L\leftarrow -\lceil \log_4(R/4)\rceil$, $H\leftarrow \lceil \log_4 N\rceil$,
$sMax\leftarrow 2\cdot 2^H+1$.

\STATE $C_I\leftarrow cell\{4,H-L+2,sMax\}$.

\FOR{$\iota = 11, 12, 21, 22$}

\FOR{$j=L$ to $H$}

\STATE {\bf if} {$j<0$} {\bf then}
$tile\leftarrow 0$
{\bf else}
 $tile \leftarrow 2^j$
{\bf end if}

\FOR{$s=-tile$ to $tile$}

\STATE Let $c_{ j, s}^{\iota}$ be a submatrix of $J$ of size $\cL_j^1\times \cL_{j,s}^2$  with respect to the support
of the digital shearlet $\sigma_{j,s,0}^\iota$.

\STATE $c_{j,s}^\iota\leftarrow c_{j,s}^{\iota}{\rm.*}\overline{\sigma^\iota_{j,s,0}}$,  $c_{j,s}^\iota\leftarrow F_2^{-1} c_{j,s}^\iota$.
\COMMENT{\texttt{Windowing with 2D iFFT.}}

\STATE $C_I\{\iota,j,s\}\leftarrow c_{j,s}^\iota$.
\ENDFOR

\ENDFOR

\ENDFOR

\STATE Let $\varphi_0^1,\varphi_0^2$ be the shearlets associated with the low-frequency part. Let $c^i, i=1,2$ denote the submatrix
of $J$ with respect to the support of $\varphi_0^i$, $i=1,2$.
\STATE $c^i \leftarrow c^i{\rm .*}\overline{\varphi^i_0}$, $i=1,2$.
\STATE $C_I\{i,L-1,0\} \leftarrow c^i$, $i=1,2$.
\end{algorithmic}
}
\end{algorithm}
}


\renewcommand{\thealgorithm}{\thesection.2}

{\allowdisplaybreaks
\begin{algorithm}[h]
\caption{Fast Adjoint Digital Shearlet Transform (Adjoint FDST)}
\label{alg:adjoint}
{\allowdisplaybreaks
\begin{algorithmic}[1]\vspace*{0.2cm}
\item[{\rm(a)}]{\bf Input:} Digital shearlet coefficients $C_I = \{c^{\iota_0}:\iota\}\cup\{c_{j,s}^{\iota} : j, s,  \iota\}$,
where each $c_{j,s}^\iota$ is a matrix of size $\cL_{j}^1\times \cL_{j,s}^2$ (see Definition~\ref{defi:digitalshearlets}), the oversampling rate
$R$, and the precomputed weight matrix $w$ of size $2\times (RN+1)\times (N+1)$.\\[1ex]

\item[{\rm(b)}] {\bf Output:} An image $I=\{[I]_{u,v}: -N/2\le u,v\le N/2-1\}$ of size $N\times N$.\\[1ex]

\item[{\rm(c)}]{\bf  Adjoint Subband Windowing:}  Subband windowing by $\cD\cS\cH$.

\STATE $L\leftarrow -\lceil \log_4(R/4)\rceil$, $H\leftarrow \lceil \log_4 N\rceil$,
$J\leftarrow zeros(2,RN+1,N+1)$.

\STATE $c^i\leftarrow C_I\{i,L-1,0\}$, $i=1,2$.

\STATE $c^i \leftarrow c^i{\rm .*}\varphi_0^i$, $i=1,2$.

\STATE Assign $c^i$ to $J$ with respect to the support of $\varphi_0^i$, $i=1,2$.


\FOR{$\iota = 11, 12, 21, 22$}

\FOR{$j=L$ to $H$}

\STATE {\bf if} {$j<0$} {\bf then}
$tile\leftarrow 0$
{\bf else}
 $tile \leftarrow 2^j$
{\bf end if}

\FOR{$s=-tile$ to $tile$}

\STATE $c_{j,s}^\iota\leftarrow C_I\{\iota,j,s\}$.

\STATE $c_{j,s}^\iota\leftarrow F_2 c_{j,s}^\iota$.\COMMENT{\texttt{the 2D-FFT.}}

\STATE $c_{j,s}^\iota\leftarrow c_{j,s}^{\iota}{\rm.*}\sigma^\iota_{j,s,0}$.

\STATE Assign $c_{j,s}$ to $J$ with respect to the support of $\sigma^\iota_{j,s,0}$.

\ENDFOR

\ENDFOR

\ENDFOR\\[1ex]

\item[{\rm(d)}]{\bf Weighting:} Weighting with the precomputed weight $w$.
\STATE $J \leftarrow J{\rm.*}\sqrt{w}$. \COMMENT{\texttt{.* is the point-wise multiplication}.}\\[1ex]

\item[{\rm(e)}] {\bf Adjoint PPFT:} The adjoint pseudo-polar Fourier transform.
\STATE $I\leftarrow zeros(N,N)$, $I_0 \leftarrow zeros(N,N)$,$m\leftarrow RN+1$,
 $J_1\leftarrow zeros(m,N)$.

\FOR{$sector = 1$ to $2$}

\FOR{$k=-RN/2$ to $RN/2$}
\STATE $q \leftarrow [J]_{sector,k,:}$.
 $\alpha \leftarrow -\frac{k}{mN/2}$.
 $q\leftarrow F_{N+1}^{-\alpha} q$.

\STATE $[J_1]_{k,:}\leftarrow E_{N+1,N}^\star q$.
\ENDFOR

\FOR{$v=-N/2$ to $N/2-1$}
\STATE $q\leftarrow [J_1]_{:,v}$.
 $q\leftarrow F_1q$, $q\leftarrow E_{m,N}^\star q$.

\STATE $[I_0]_{:,v} \leftarrow q$.
\ENDFOR

\STATE $I = I+I_0$.
\ENDFOR

\end{algorithmic}
}
\end{algorithm}
}


\renewcommand{\thealgorithm}{\thesection.3}

{\allowdisplaybreaks
\begin{algorithm}[h]
\caption{Inverse Fast Digital Shearlet Transform (Inverse FDST)}
\label{alg:inverse}
{\allowdisplaybreaks
\begin{algorithmic}[1]\vspace*{0.2cm}
\item[{\rm(a)}]{\bf Input:} Digital shearlet coefficients $C_I= \{c^{\iota_0}:\iota_0\}\cup\{c_{j,s}^{\iota} : j, s,  \iota\}$,
the oversampling rate $R$, the precomputed weight matrix $w$ of size $2\times (RN+1)\times (N+1)$, an initial guess $I_0$, a precision
parameter $\varepsilon$, and a maximal iteration number $itMax$.\\[1ex]

\item[{\rm(b)}] {\bf Output:} An image $I=\{[I]_{u,v}: -N/2\le u,v\le N/2-1\}$ of size $N\times N$.\\[1ex]

\item[{\rm (c)}]{\bf Initialization:} Perform the adjoint shearlet windowing on $C_I$ by using Algorithm~2 to obtain an image $b$, say.\\[1ex]

\item[{\rm(d)}]{\bf  CG Iteration:} Let $A = P^\star w P$ generated by Algorithm \ref{alg:ppft}, Algorithm \ref{alg:forward}(e), and Algorithm
\ref{alg:adjoint}(e). We apply the CG method to solve the linear system $A I = b$.
\STATE $r_0 \leftarrow b-A I_0$, $p_0 \leftarrow r_0$, $k \leftarrow 0$.
\WHILE{$\|r_k\|_2>\varepsilon$ and $k< itMax$}
\STATE $\alpha_k \leftarrow \frac{r_k^Tr_k}{p_k^T A p_k}$.
\STATE $I_{k+1}\leftarrow I_k + \alpha_k p_k$.
\STATE $r_{k+1}\leftarrow r_k-\alpha_k A p_k$.
\STATE $\beta_{k}\leftarrow \frac{r_{k+1}^Tr_{k+1}}{r_k^Tr_k}$.
\STATE $p_{k+1}\leftarrow r_{k+1}+\beta_k  p_k$.
\STATE $k\leftarrow k+1$.
\ENDWHILE
\STATE $I\leftarrow I_k$.
\end{algorithmic}
}
\end{algorithm}
}


\end{document}